\documentclass[11pt,a4paper]{article}

\usepackage{exscale,makeidx,amssymb,amsmath,bbm,color,psfrag}
\usepackage{amsmath}
\usepackage{graphics}
\usepackage{amsthm,amsfonts}
\usepackage[in]{fullpage}

\newcommand{\E}[1]{\ensuremath{\mathbb{E} \! \left[#1 \right]}}
\newcommand{\Prob}[1]{\ensuremath{\mathbb{P} \! \left(#1 \right)}}

\newcommand{\I}[1]{\ensuremath{\mathbbm{1}_{ \{ #1 \} }}}
\newcommand{\R}{\ensuremath{\mathbb{R}}}

\newcommand{\N}{\ensuremath{\mathbb{N}}}

\newcommand{\fl}[1]{\ensuremath{\lfloor #1 \rfloor}}

\renewcommand{\subset}{\subseteq}
\newcommand{\convdist}{\ensuremath{\stackrel{d}{\rightarrow}}}

\newcommand{\equidist}{\ensuremath{\stackrel{d}{=}}}

\newcommand{\argmax}{\ensuremath{\operatornamewithlimits{argmax}}}

\newcommand{\Sfl}{\ensuremath{\mathcal{S}^{\downarrow}_1}}
\newcommand{\Sflfin}{\ensuremath{\mathcal{S}^{\downarrow}}}

\newcommand{\convlaw}{\ensuremath{\stackrel{d}{\rightarrow}}}

\renewcommand{\epsilon}{\varepsilon}

\newtheorem{theorem}{Theorem}[section]
\newtheorem{definition}[theorem]{Definition}
\newtheorem{proposition}[theorem]{Proposition}
\newtheorem{corollary}[theorem]{Corollary}
\newtheorem{lemma}[theorem]{Lemma}

\makeatletter
\renewenvironment{proof}[1][\proofname]{\par
  \pushQED{\qed}%
  \normalfont \topsep6\p@\@plus6\p@\relax
  \trivlist
  \item[\hskip\labelsep
        \bf
    #1\@addpunct{.}]\ignorespaces
}{%
  \popQED\endtrivlist\@endpefalse
}
\makeatother

\numberwithin{equation}{section}

\date{}

\begin{document}

\title{Behavior near the extinction time in self-similar
  fragmentations I: the stable case} \author{Christina Goldschmidt
  \thanks{Department of Statistics, University of Oxford;
    \texttt{goldschm@stats.ox.ac.uk}} \and B\'en\'edicte Haas
  \thanks{Universit\'e Paris-Dauphine, Ceremade, F-75016 Paris, France;
    \texttt{haas@ceremade.dauphine.fr}}}

\maketitle

%%%%%%%%%%%%
% Abstract %
%%%%%%%%%%%%

\vspace{-0.5cm}
\begin{abstract}
  \noindent The stable fragmentation with index of self-similarity
  $\alpha \in [-1/2,0)$ is derived by looking at the masses of the
  subtrees formed by discarding the parts of a $(1 +
  \alpha)^{-1}$--stable continuum random tree below height $t$, for $t
  \geq 0$.  We give a detailed limiting description of the
  distribution of such a fragmentation, $(F(t), t \geq 0)$, as it
  approaches its time of extinction, $\zeta$.  In particular, we show
  that $t^{1/\alpha}F((\zeta - t)^+)$ converges in distribution as $t
  \to 0$ to a non-trivial limit.  In order to prove this, we go
  further and describe the limiting behavior of (a) an excursion of
  the stable height process (conditioned to have length 1) as it
  approaches its maximum; (b) the collection of open intervals where
  the excursion is above a certain level and (c) the ranked sequence
  of lengths of these intervals.  Our principal tool is excursion
  theory.  We also consider the last fragment to disappear and show
  that, with the same time and space scalings, it has a limiting
  distribution given in terms of a certain size-biased version of the
  law of $\zeta$.

  In addition, we prove that the logarithms of the sizes of the
  largest fragment and last fragment to disappear, at time $(\zeta -
  t)^+$, rescaled by $\log(t)$, converge almost surely to the constant
  $-1/\alpha$ as $t \to 0$.
\end{abstract}

\renewcommand{\abstractname}{R\'esum\'e}
\begin{abstract}
  \noindent La fragmentation stable d'incice $\alpha \in [-1/2,0)$ est
  construite \`a partir des masses des sous-arbres de l'arbre continu
  al\'eatoire stable d'indice $(1+\alpha)^{-1}$ obtenus en ne gardant
  que les feuilles situ\'ees \`a une hauteur sup\'erieure \`a $t$, pour
  $t \geq 0$. Nous donnons une description d\'etaill\'ee du
  comportement asymptotique d'une telle fragmentation, $(F(t), t \geq
  0)$, au voisinage de son point d'extinction, $\zeta$. En
  particulier, nous montrons que $t^{1/\alpha} F((\zeta-t)^+)$
  converge en loi lorsque $t \rightarrow 0$ vers une limite non
  triviale. Pour obtenir ce r\'esultat, nous allons plus loin et
  d\'ecrivons le comportement asymptotique en loi, apr\`es
  normalisation, (a) d'une excursion du processus de hauteur stable
  (conditionn\'ee \`a avoir une longueur $1$) au voisinage de son
  maximum; (b) des intervalles ouverts o\`u l'excursion est au-dessus
  d'un certain niveau; et (c) de la suite d\'ecroissante des longueurs
  de ces intervalles. Notre outil principal est la th\'eorie des
  excursions. Nous nous int\'eressons \'egalement au dernier fragment
  \`a dispara\^itre et montrons, qu'avec les m\^emes normalisations en
  temps et espace, la masse de ce fragment a une distribution limite
  contruite \`a partir d'une certaine version biais\'ee de $\zeta$.
 
  Enfin, nous montrons que les logarithmes des masses du plus gros
  fragment et du dernier fragment \`a dispara\^itre, au temps
  $(\zeta-t)^+$, divis\'es par $\log(t)$, convergent presque
  s\^urement vers la constante $-1/\alpha$ lorsque $t \rightarrow 0$.
\end{abstract}

\noindent \emph{AMS subject classifications: 60G18, 60G52, 60J25 
}

%%%%%%%%%%%%%%%%
% Introduction %
%%%%%%%%%%%%%%%%

\section{Introduction}

The subject of this paper is a class of random fragmentation processes
which were introduced by Bertoin~\cite{BertoinSSF}, called the
self-similar fragmentations.  In fact, we will find it convenient to
have two slightly different notions of a fragmentation process.  By an
\emph{interval fragmentation}, we mean a process $(O(t), t \geq 0)$
taking values in the space of open subsets of $(0,1)$ such that $O(t)
\subset O(s)$ whenever $0 \leq s \leq t$.  We refer to a connected
interval component of $O(t)$ as a \emph{block}.  Let $F(t) = (F_1(t),
F_2(t), \ldots)$ be an ordered list of the lengths of the blocks of
$O(t)$.  Then $F(t)$ takes values in the space 
\[
\Sfl = \left\{\mathbf{s} = (s_1, s_2, \ldots): s_1 \geq s_2 \geq \ldots
\geq 0, \sum_{i=1}^{\infty} s_i \leq 1\right\}. 
\]
We call the process $(F(t), t \geq 0)$ a \emph{ranked fragmentation}.
A ranked fragmentation is called \emph{self-similar with index $\alpha
\in \R$} if it is a time-homogeneous Markov process which satisfies
certain \emph{branching} and \emph{self-similarity} properties.
Roughly speaking, these mean that every block should split into a
collection of sub-blocks whose relative lengths always have the same
distribution, but at a rate which is proportional to the length of the
original block raised to the power $\alpha$.  (Rigorous definitions
will be given in Section~\ref{sec:preliminaries}.)  Clearly the sign
of $\alpha$ has a significant effect on the behavior of the process.
If $\alpha > 0$ then larger blocks split faster than smaller ones,
which tends to act to balance out block sizes.  On the other hand, if
$\alpha < 0$ then it is the smaller blocks which split faster.
Indeed, small blocks split faster and faster until they are reduced to
\emph{dust}, that is blocks of size 0.

The asymptotic behavior of self-similar fragmentations has been
studied quite extensively.  In one sense, it is trivial, in that $F(t)
\to 0$ a.s.\ as $t \to \infty$, whatever the value of the index
$\alpha$ (provided the process is not trivially constant, i.e.\ equal
to its initial value for all times $t$).  For $\alpha \geq 0$,
rescaled versions of the empirical measures of the lengths of the
blocks have law of large numbers-type behavior (see
Bertoin~\cite{BertoinAB} and Bertoin and
Rouault~\cite{Bertoin/Rouault}).  For $\alpha < 0$, however, the
situation is completely different.  Here, there exists an almost
surely finite random time $\zeta$, called the \emph{extinction time},
when the state is entirely reduced to dust.  The manner in which mass
is lost has been studied in detail in \cite{HaasLossMass} and
\cite{HaasRegularity}.

The purpose of this article is to investigate the following more
detailed question when $\alpha$ is negative: \textbf{how does
  the process $F((\zeta-t)^+)$ behave as $t \to 0$ ?}  We provide a
detailed answer for a particularly nice one-parameter family of
self-similar fragmentations with negative index, called the
\emph{stable fragmentations}.

The simplest of the stable fragmentations is the \emph{Brownian
  fragmentation}, which was first introduced and studied by
Bertoin~\cite{BertoinSSF}.  Suppose that $(\mathbf{e}(x), 0 \leq x \leq 1)$ is
a standard Brownian excursion.  Consider, for $t \geq 0$ the sets
\[
O(t) = \{x \in [0,1]: \mathbf{e}(x) > t\}
\]
and let $F(t) \in \Sfl$ be the lengths of the interval components of
$O(t)$ in decreasing order.  Then it can be shown that $(F(t), t \geq
0)$ is a self-similar fragmentation with index $-1/2$.
Miermont~\cite{Miermont} generalized this construction by replacing
the Brownian excursion with an excursion of the height process
associated with the stable tree of index $\beta \in (1,2)$, introduced
and studied by Duquesne, Le Gall and Le Jan
\cite{Duquesne/LeGall,LeGall/LeJan}.  The corresponding process is a
self-similar fragmentation of index $\alpha = 1/\beta - 1$.

The behavior near the extinction time in the Brownian fragmentation
can be obtained via a decomposition of the excursion at its maximum.
We discuss this case in Section~\ref{sec:Brownian}.  Abraham and
Delmas~\cite{Abraham/Delmas} have recently proved a generalized
Williams' decomposition for the excursions which code stable trees.
This provides us with the necessary tools to give a complete
description of the behavior of the stable fragmentations near their
extinction time, which is detailed in Section~\ref{sec:stable}.  In
every case, we obtain that
\[
t^{1/\alpha}F((\zeta-t)^{+}) \convlaw F_{\infty} \text{ as }t \rightarrow 0,
\]
where $F_{\infty}$ is a random limit which takes values in the
set of non-increasing non-negative sequences with finite sum. The
limit $F_{\infty}$ is constructed from a self-similar function
$H_{\infty}$ on $\R$, which itself arises when looking at the scaling
behavior of the excursion in the neighborhood of its maximum.  See
Theorems \ref{thm:stablefragfonc} and \ref{thm:stablefrag} and
Corollary \ref{corostablerank} for precise statements.

In Corollary \ref{lastfrag}, we also consider the process of the
\emph{last fragment}, that is the size of the (as it turns out unique)
block which is the last to disappear.  We call this size $F_{\ast}(t)$ and prove that, scaled as
before, $F_{\ast}((\zeta-t)^+)$ also has a limit in distribution as $t \to 0$ which,
remarkably, can be expressed in terms of a certain size-biased version
of the distribution of $\zeta$.

Sections~\ref{sec:technicalbackground}, \ref{sec:convheightprocess},
\ref{sec:proof} and \ref{seclastfrag} are devoted to the proofs of
these results.

Finally, in Section~\ref{SectionLog}, we consider the logarithms of
the largest and last fragments and show that
\[
\frac{\log F_1((\zeta - t)^+)}{\log(t)} \to -1/\alpha \quad
\text{and} \quad \frac{\log F_{\ast}((\zeta - t)^+)}{\log(t)} \to -1/\alpha
\]
almost surely as $t \to 0$.  In fact, these results hold for a more
general class of self-similar fragmentations with negative index.

We will investigate the limiting behavior of $F((\zeta -t)^{+})$ as $t
\to 0$ for general self-similar fragmentations with negative index
$\alpha$ in future work, starting in \cite{Goldschmidt/Haas2}.  In
general, as indicated by the results for the logarithms of $F_1$ and
$F_{\ast}$ above, the natural conjecture is that $t^{1/\alpha}$ is the
correct re-scaling for non-trivial limiting behavior.  However, since
the excursion theory tools we use here are not available in general,
we are led to develop other methods of analysis.

%%%%%%%%%%%%%%%%%%%%%%%%%%%%%%%
% Self-similar fragmentations %
%%%%%%%%%%%%%%%%%%%%%%%%%%%%%%%

\section{Self-similar fragmentations}
\label{sec:preliminaries}

Define $\mathcal{O}_{(0,1)}$ to be the set of open subsets of $(0,1)$.
We begin with a rather intuitive notion of a fragmentation process.

\begin{definition}[Interval fragmentation]
  An interval fragmentation is a process $(O(t), t \geq 0)$ taking
  values in $\mathcal{O}_{(0,1)}$ such that $O(t) \subset O(s)$ whenever $0
  \leq s \leq t$.
\end{definition}

In this paper we will be dealing with interval fragmentations which
derive from excursions.  Here, an \emph{excursion} is a continuous
function $f:[0,1] \to \R^+$ such that $f(0) = f(1) = 0$ and $f(x) > 0$
for all $x \in (0,1)$.  The associated interval fragmentation, $(O(t),
t \geq 0)$ is defined as follows: for each $t \geq 0$,
\[
O(t) = \{x \in [0,1]: f(x) > t\}.
\]
An example is given in Figure~\ref{fig:schematic}.

\begin{figure}
\begin{center}
\psfrag{t}{$t$}
\psfrag{0}{0}
\psfrag{1}{1}
\includegraphics{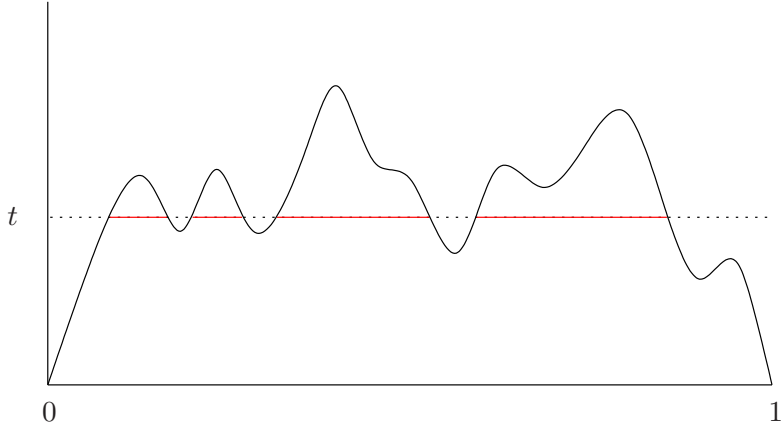}
\end{center}
\caption{An interval fragmentation derived from a continuous
  excursion: the set $O(t)$ is represented by the solid lines at level
  $t$.}
\label{fig:schematic}
\end{figure}

We need to introduce a second notion of a fragmentation process.  We
endow the space $\Sfl$ with the topology of pointwise convergence.

\begin{definition}[Ranked self-similar fragmentation]
  A ranked self-similar fragmentation $(F(t), t \geq 0)$ with index
  $\alpha \in \R$ is a c\`adl\`ag Markov process taking values in
  $\Sfl$ such that
\begin{itemize}
\item $(F(t), t \geq 0)$ is continuous in probability;
\item $F(0) = (1,0,0, \ldots)$;
\item Conditionally on $F(t) = (x_1, x_2, \ldots)$, $F(t+s)$ has the
  law of the decreasing rearrangement of the sequences $x_i
  F^{(i)}(x_i^{\alpha} s)$, $i \geq 1$, where $F^{(1)}, F^{(2)},
  \ldots$ are independent copies of the original process $F$.
\end{itemize}
\end{definition}

Let $r: \mathcal{O}_{(0,1)} \to \Sfl$ be the function which to an open
set $O \subset (0,1)$ associates the ranked sequence of the lengths of
its interval components.  We say that an interval fragmentation is
\emph{self-similar} if it possesses branching and self-similarity
properties which entail that $(r(O(t)), t \geq 0)$ is a ranked
self-similar fragmentation.  See \cite{Basdevant,BertoinSSF,BertoinBook} for
this and further background material.

Bertoin~\cite{BertoinSSF} has proved that a ranked self-similar
fragmentation can be characterized by three parameters, $(\alpha, \nu,
c)$.  Here, $\alpha \in \R$; $\nu$ is a measure on $\Sfl$ such that
$\nu((1,0,\ldots)) = 0$ and $\int_{\Sfl} (1 - s_1) \nu(\mathrm{d}\mathbf{s}) <
\infty$; and $c \in \R^+$.  The parameter $\alpha$ is the \emph{index
  of self-similarity}.  The measure $\nu$ is the \emph{dislocation
  measure}, which describes the way in which fragments suddenly
dislocate; heuristically, a block of mass $m$ splits at rate
$m^{\alpha} \nu(\mathrm{d}\mathbf{s})$ into blocks of masses $(ms_1, ms_2,
\ldots)$.  The real number $c$ is the \emph{erosion coefficient},
which describes the rate at which blocks continuously melt.  Note that
$\nu$ may be an infinite measure, in which case the times at which
dislocations occur form a dense subset of $\R^+$.  When $c=0$, the
fragmentation is a pure jump process.

In the context of an interval fragmentation derived from an excursion,
it is easy to see that the extinction time of the fragmentation is
just the maximum height of the excursion:
\[
\zeta = \max_{0 \leq x \leq 1} f(x).
\]
In the examples we treat, this maximum will be attained at a unique
point, $x_*$.  In this case, let $O_*(t)$ be the interval component of
$O(t)$ containing $x_*$ at time $t$, and let $F_{\ast}(t)$ be its
length, i.e.\ $F_*(t) = |O_*(t)|$.  We call both $O_*$ and $F_*$ the
\emph{last fragment process}.

%%%%%%%%%%%%%%%%%%%%%%%%%%%%%%
% The Brownian fragmentation %
%%%%%%%%%%%%%%%%%%%%%%%%%%%%%%

\section{The Brownian fragmentation}
\label{sec:Brownian}

We begin by discussing the special case of the Brownian fragmentation.
The sketch proofs in this section are not rigorous, but can be made
so, as we will demonstrate later in the paper.  Our intention is to
introduce the principal ideas in a framework which is familiar to the
reader.

Let $(\mathbf{e}(x), 0 \leq x \leq 1)$ be a normalized Brownian
excursion with length 1 and for each $t \geq 0$ define the associated
interval fragmentation by
\begin{equation*}
O(t):=\left \{x \in [0,1] : \mathbf{e}(x)>t \right\}.
\end{equation*}
See Figure~\ref{fig:Brownianfrag} for a picture.
\begin{figure}
\begin{center}
\psfrag{t}{$t$}
\psfrag{z}{$\zeta$}
\psfrag{0.0}{0.0}
\psfrag{0.2}{0.2}
\psfrag{0.4}{0.4}
\psfrag{0.6}{0.6}
\psfrag{0.8}{0.8}
\psfrag{1.0}{1.0}
\psfrag{1.2}{1.2}
\rotatebox{270}{\resizebox{10cm}{!}{\includegraphics{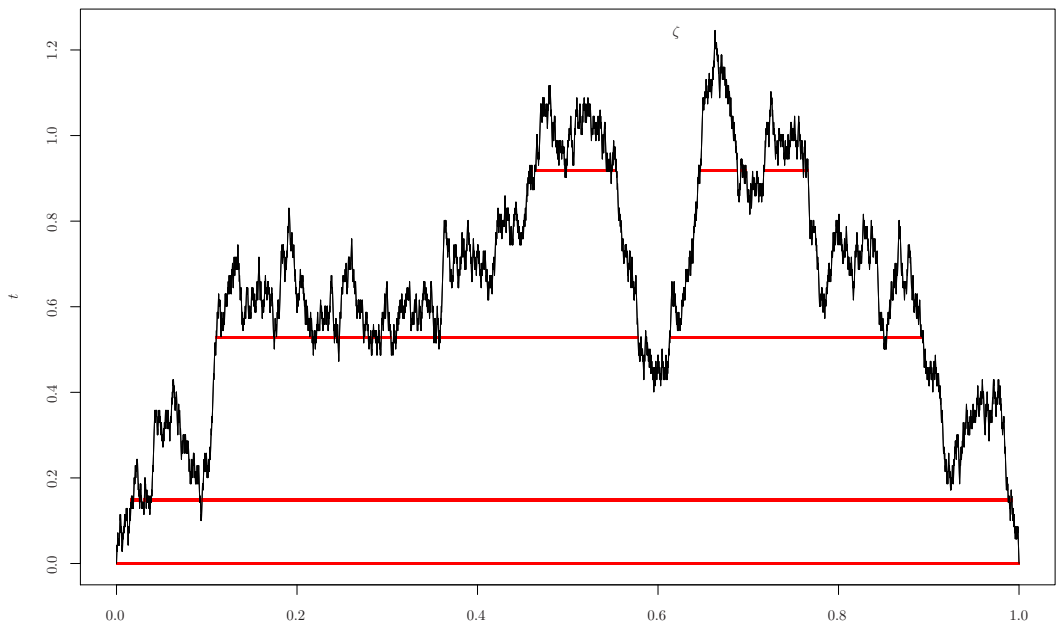}}}
\caption{The Brownian interval fragmentation with the open intervals
  which constitute the state at times $t = 0, 0.15, 0.53$ and $0.92$
  indicated.}
\label{fig:Brownianfrag}
\end{center}
\end{figure}
The associated ranked fragmentation process, $(F(t), t \geq 0)$, has
index of self-similarity $\alpha=-1/2$, binary dislocation measure
specified by $\nu(s_1 + s_2 < 1) = 0$ and 
\[
\nu(s_1 \in \mathrm dx) = 2(2 \pi x^3 (1-x)^3)^{-1/2}
\mathbbm{1}_{[1/2,1)}(x) \mathrm dx, 
\]
and erosion coefficient $0$.  See Bertoin
\cite{BertoinSSF} for a proof.  The extinction time of this
fragmentation process is the maximum of the Brownian excursion.  In
particular, from Kennedy~\cite{Kennedy} we have
\[
\Prob{\zeta > t} = 2 \sum_{n=1}^{\infty} (4t^2n^2 - 1) \exp(-2t^2
n^2), \text{  } t\geq 0.
\]
To the best of our knowledge, this is the only fragmentation for which
the law of $\zeta$ is known.  It is well known that the maximum of
$\mathbf{e}$ is reached at a unique point $x_{*} \in [0,1]$ a.s., and
so the mass, $F_{\ast}(t)$, of the last fragment to survive is
well-defined.

There is a complete characterization of the limit in law of the
rescaled fragmentation near to its extinction time.

\begin{theorem}[Brownian fragmentation] \label{thm:Brownianfrag}
If $O$ is the Brownian interval fragmentation then
\begin{equation*}
\frac{O \left((\zeta -t \right)^{+})-x_{*}}{t^{2}}
\convlaw O_{\infty} \text{ as
$t\rightarrow 0$},
\end{equation*}
where $O_{\infty}=\{x\in
\mathbb{R}:H_{\infty}(x)<1\}$,
\begin{equation*}
H_{\infty}(x)=R_+(x)\I{x\geq 0}+ R_{-}(-x)\I{x<0}
\end{equation*}
and $R_+$ and $R_{-}$ are two independent $\mathrm{Bes}(3)$
processes.
\end{theorem}

A full discussion of the topology in which the above convergence in
distribution occurs is deferred until Section
\ref{subsec:topology}.  A proof of this theorem is given in
Uribe Bravo~\cite{Uribe}.  We give here a sketch, since a rigorous proof
will be given in the wider setting of general stable fragmentations.

\begin{proof}[Sketch of proof.]
We decompose the excursion $\mathbf{e}$ at its maximum $\zeta$.
Define
\[
\tilde{\mathbf{e}}(x) 
= \begin{cases}
\zeta - \mathbf{e}(x_* + x), & 0 \leq x \leq 1-x_* \\
\zeta - \mathbf{e}(x - 1 + x_*), & 1-x_* < x \leq 1.
\end{cases}
\]
Then by Williams' decomposition for the Brownian excursion
\cite[Section VI.55]{Rogers/Williams2}, we have that
$\tilde{\mathbf{e}}$ is again a standard Brownian excursion.
Moreover, if $t < \zeta$ then 
\begin{align*}
& t^{-2}(O(\zeta-t) - x_*) \\
& \qquad = \{y \in [0,(1-x_*)t^{-2}]:
\tilde{\mathbf{e}}(yt^2) < t\} \cup \{y \in [-x_* t^{-2},0]:
\tilde{\mathbf{e}}(1 + yt^2) < t\}.
\end{align*}
Now by the scaling property of Brownian motion, $(t^{-1}
\tilde{\mathbf{e}}(xt^2), 0 \leq x \leq t^{-2})$ has the distribution of a
Brownian excursion of length $t^{-2}$, which we will denote $(b^{t}(x), 0
\leq x \leq t^{-2})$.  So the above set has the same distribution as
\[
\{x \in [0, (1-x_*)t^{-2}]: b^{t}(x) < 1\} \cup \{x \in
[-x_*t^{-2},0]: b^{t}(t^{-2} + x) < 1\}.
\]
Fix $n \in \R^+$.  As $t \to 0$, the length of the excursion goes to
$\infty$ and $(b^{t}(x), 0 \leq x \leq n) \convdist (R_+(x), 0 \leq x
\leq n)$, where $(R_+(x), x \geq 0)$ is a 3-dimensional Bessel process
started at 0 (see, for example, Theorem 0.1 of
Pitman~\cite{PitmanStFl}).  Moreover, by symmetry, $(b^{t}(t^{-2} -
x), 0 \leq x \leq n) \convdist (R_{-}(x), 0 \leq x \leq
n)$, where $R_{-}$ is (in fact) an independent copy of $R_+$.  Thus we
obtain
\[
t^{-2}(O((\zeta-t)^+) - x_*) \convdist O_{\infty},
\]
as claimed.
\end{proof}

Theorem~\ref{thm:Brownianfrag} enables us to deduce an explicit limit
law for the associated ranked fragmentation. See Corollary
\ref{corostablerank} for a precise statement and note that, as
detailed at the end of Section \ref{sec:stable}, the passage from the
convergence of open sets to that of these ranked sequences is not
immediate.  We also have the following limit law for the last
fragment, $F_{\ast}$.

\begin{corollary}
If $F$ is the Brownian fragmentation then
\begin{equation*}
\frac{F_{\ast }((\zeta - t)^{+})}{t^{2}} \convlaw \left(\frac{2\zeta
  }{\pi }\right) ^{2} \text{ as $t \to 0$,}
\end{equation*}
or, equivalently,
\begin{equation*}
t (F_{\ast} ( (\zeta - t)^{+} ) )^{-1/2}
\convlaw \zeta_{\ast},
\end{equation*}
where $\zeta _{\ast}$ is a size-biased version of $\zeta$, i.e.\
$\E{f(\zeta_{\ast})} = \E{\zeta f(\zeta)}/ \E{\zeta}$ for every
test function $f$.
\end{corollary}

\begin{proof}
Let $T_+ = \inf\{t \geq 0: R_+(t) = 1\}$ and $T_{-} = \inf\{t
\geq 0: R_{-}(t) = 1\}$, where $R_+$ and $R_{-}$ are
the independent $\mathrm{Bes(3)}$ processes from the statement of
Theorem~\ref{thm:Brownianfrag}.  Then by
Theorem~\ref{thm:Brownianfrag},
\[
\frac{F_{\ast}((\zeta - t)^{+})}{t^{2}} \convlaw T_+ + T_{-}.
\]
By Proposition 2.1 of Biane, Pitman and Yor~\cite{Biane/Pitman/Yor},
\[
T_+ + T_{-} \equidist \left(\frac{2 \zeta}{\pi}\right)^2
\]
and, moreover, if we define $Y = \sqrt{\frac{2}{\pi}} \zeta$ then
$Y$ satisfies
\[
\E{f(1/Y)} = \E{Y f(Y)}
\]
for any test function $f$ (in particular, $Y$ has mean 1). Hence,
\[
\E{f\left(\frac{\pi}{2 \zeta}\right)} = \sqrt{\frac{2}{\pi}}
\E{\zeta f(\zeta)} = \E{\zeta f(\zeta)} / \E{\zeta},
\]
which completes the proof.
\end{proof}

\textbf{Remark 1.}  As noted by Uribe Bravo~\cite{Uribe}, the random
variable $(2 \zeta / \pi)^2$ has Laplace transform $2 \lambda(\sinh
\sqrt{2 \lambda})^{-2}$.  He also uses another result in Biane, Pitman
and Yor~\cite{Biane/Pitman/Yor} to show that the Lebesgue measure of
the set $O_{\infty}$ has Laplace transform $(\cosh \sqrt{2
  \lambda})^{-2}$.  Let $M(t)$ be the total mass of the fragmentation
at time $t$, that is the Lebesgue measure of $O(t)$.  Then this
entails that
\[
\frac{M((\zeta-t)^+)}{t^2} \convdist M_{\infty},
\]
where $M_{\infty}$ has Laplace transform $(\cosh \sqrt{2 \lambda})^{-2}$.

\textbf{Remark 2.} The Bes(3) process encodes a fragmentation process
with immigration which arises naturally when studying rescaled
versions of the Brownian fragmentation near $t=0$ (see Haas
\cite{HaasAsympFrag}).  This is closely related to our approach: using
Williams' decomposition of the Brownian excursion, we obtain results
on the behavior of the fragmentation near its extinction time by
studying the sets of $\{x\in[0,1]: \mathbf{e}(x)<t\}$ for small $t$. This
duality between the behavior of the fragmentation near $0$ and near
its extinction time seems to be specific to the Brownian case.

%%%%%%%%%%%%%%%%%%%%%%%%%%%%%%%%%
% General stable fragmentations %
%%%%%%%%%%%%%%%%%%%%%%%%%%%%%%%%%

\section{General stable fragmentations}
\label{sec:stable}

There is a natural family which generalizes the Brownian
fragmentation: the \emph{stable fragmentations}, constructed and
studied by Miermont in \cite{Miermont}. The starting point is the
stable height processes $H$ with index $1<\beta\leq2$ which were
introduced by Duquesne, Le Gall and Le Jan
\cite{Duquesne/LeGall,LeGall/LeJan} in order to code the genealogy of
continuous state branching processes with branching mechanism $\lambda^{\beta}$ via stable trees. We do
not give a definition of these processes here, since it is rather
involved; full definitions will be given in the course of the next
section.  Here, we simply recall that it is possible to consider a
normalized excursion of $H$, say $\mathbf{e}$, which is almost surely
continuous on $[0,1]$. When $\beta=2$, this is the normalized Brownian
excursion (up to a scaling factor of $\sqrt{2}$).

Once again, let
\begin{equation*}
O(t):=\left \{x \in [0,1] : \mathbf{e}(x)>t \right\}.
\end{equation*}
For $1<\beta<2$, Miermont \cite{Miermont} proved that the corresponding ranked
fragmentation is a self-similar fragmentation of index
$\alpha=1/\beta-1$ and erosion coefficient $0$.  The dislocation
measure is somewhat harder to express than that of the Brownian
fragmentation.  Let $T$ be a stable subordinator of Laplace exponent
$\lambda^{1/\beta}$ and write $\Delta T_{[0,1]}$ for the sequence of
its jumps before time 1, ranked in decreasing order.  Then for any
non-negative measurable function $f$,
\[
\int_{\Sfl}f(\mathbf s)\nu(\mathrm d \mathbf s) =
\frac{\beta(\beta-1) \Gamma(1 - \frac{1}{\beta})}{\Gamma(2 -
  \beta)} \E{T_1 f(T_1^{-1} \Delta T_{[0,1]})}.
\]

As we will discuss in Section~\ref{subsec:heightprocesses}, there
is a unique point $x_* \in [0,1]$ at which $\mathbf{e}$ attains its
maximum (this maximum is denoted $\zeta$ to be consistent with earlier
notation,  so that $\zeta = \mathbf{e}(x_{\ast})$).  So the size of the last fragment, $F_*$, is well-defined
for the stable fragmentations.  We first state a result on the
behavior of the stable height processes near their maximum.
\begin{theorem}
\label{thm:stablefragfonc}
Let $\mathbf e$ be a normalized excursion of the stable height process
with parameter $\beta$ and extend its definition to $\mathbb R$ by setting
$\mathbf{e}(x)=0$ when $x \notin [0,1]$.  Then there exists an almost surely
positive continuous self-similar function $H_{\infty}$ on
$\mathbb{R}$, which is symmetric in distribution (in the sense that
$(H_\infty(-x), x \geq 0) \equidist (H_\infty(x), x \geq 0)$) and
converges to $+\infty$ as $x \to +\infty$ or $x \to -\infty$, and
which is such that
\[
t^{-1}\left(\zeta-\mathbf e(x_*+t^{-1/\alpha} \ \cdot \ ) \right) \convlaw
H_{\infty} \text{ as $t \to 0$},
\]
where $\alpha=1/\beta-1$. The convergence holds
with respect to the topology of uniform convergence on compacts.
\end{theorem} 

A precise definition of $H_{\infty}$ is given in Section~\ref{sec:Williams}.  Intuitively, we can think of it as an excursion of the height process of infinite length, centered at its ``maximum" and flipped upside down.

Theorem~\ref{thm:stablefragfonc} leads to the following generalization of
Theorem~\ref{thm:Brownianfrag}.

\begin{theorem}[Stable interval fragmentation] \label{thm:stablefrag} 
  Let $O$ be a stable interval fragmentation with parameter $\beta$ and
  consider the corresponding self-similar function $H_{\infty}$
  introduced in Theorem \ref{thm:stablefragfonc}. Then
\begin{equation*}
t^{1/\alpha}\left(O\left((\zeta - t)^{+}\right)-x_{*}\right) 
\convlaw \{x\in \mathbb{R}:H_{\infty}(x)<1\} \text{ as $t\rightarrow 0$}.
\end{equation*}
\end{theorem}

The topology on the bounded open sets of $\mathbb{R}$ in which this
convergence occurs will be discussed in the next section. 

Define
\[
\Sflfin := \left\{\mathbf{s} = (s_1,s_2, \ldots): s_1 \geq
  s_2 \geq \ldots \geq 0, \sum_{i=1}^{\infty} s_i < \infty\right\}.
\]
We endow this space with the distance
\[
d_{\Sflfin}(\mathbf{s},\mathbf{s'})=\sum_{i=1}^{\infty}|s_i-s_i'|.
\]
We also have the following ranked counterpart of
Theorem~\ref{thm:stablefrag}: let $F(t)$ be the decreasing sequence of
lengths of the interval components of $(O(t),t\geq 0)$ and, similarly,
let $F_{\infty}$ be the decreasing sequence of lengths of the interval
components of $\{x\in \mathbb{R}:H_{\infty}(x)<1\}$.  Then $F_{\infty}
\in \Sflfin$.

\begin{corollary}[Ranked stable fragmentation]
\label{corostablerank}
As $t \to 0$,
\[
t^{1/\alpha}F((\zeta-t)^+)\convlaw F_{\infty}.
\]
\end{corollary}

In particular, this gives the behavior of the total mass $M(t) :=
\sum_{i = 1}^{\infty} F_i(t)$ of the fragmentation near its extinction
time.

Finally, as in the Brownian case, the distribution of the limit of the
size of the last fragment can be expressed in terms of a size-biased
version of the height $\zeta$.

\begin{corollary}[Behavior of the last fragment]
\label{lastfrag}
As $t \rightarrow 0$,
\begin{equation}
\label{eqn:laststable}
t\left(F_{\ast}\left( (\zeta -t)^{+}\right)\right)^{\alpha}
\convlaw \zeta_{\ast^{\alpha}}
\end{equation}
where $\zeta_{\ast^{\alpha}}$ is a ``$(-1/\alpha-1)$-size-biased''
version of $\zeta$, which means that for every test-function $f$,
\begin{equation}
\label{eqn:biasedzeta} \E{f(\zeta_{\ast^{\alpha}})} =\frac{
\E{\zeta^{-1/\alpha-1} f(\zeta)}}{\E{\zeta^{-1/\alpha-1}}}.
\end{equation}
Moreover,
\vspace{-5pt}
\begin{itemize}
\item[(i)] there exist positive constants $0<A<B$ such that
\[
\exp\left(-Bt^{1/(1+\alpha)}\right) \leq \Prob{\zeta_{*^{\alpha}} \geq t}
\leq \exp\left(-At^{1/(1+\alpha)}\right) 
\]
for all $t$ sufficiently large; \\
\item[(ii)] for any $q <1-1/\alpha$, 
\[
\Prob{\zeta_{*^{\alpha}} \leq t} \leq t^q,
\]
for all $t \geq 0$ sufficiently small.
\end{itemize}
\end{corollary}

The proof of Theorem \ref{thm:stablefragfonc} is based on the
``Williams' decomposition'' of the height function $H$ given by
Abraham and Delmas~\cite[Theorem 3.2]{Abraham/Delmas}, and can be
found in Section \ref{sec:convheightprocess}.  We emphasise the fact
that uniform convergence on compacts of a sequence of continuous
functions $f_n:\mathbb{R} \rightarrow \mathbb{R}$ to $f:\mathbb{R}
\rightarrow \mathbb{R}$ does not imply in general that the sets $\{x
\in \mathbb{R} : f_n(x)<1 \}$ converge to $\{x \in \mathbb{R} : f(x)<1
\}$. Take, for example, $f$ constant and equal to $1$ and $f_n$
constant and equal to $1-1/n$ (see the next section for the topology
we consider on open sets of $\mathbb{R}$). Less trivial examples show
that there may exist another kind of problem when passing from the
convergence of functions to that of ranked sequences of lengths of
interval components: take, for example, a continuous function
$f:\mathbb{R} \rightarrow \mathbb{R}^+$ which is strictly larger than
$1$ on $\mathbb{R} \backslash [-1,1]$ and then consider continuous
functions $f_n:\mathbb{R} \rightarrow \mathbb{R}^+$ such that $f_n=f$ on
$[-n,n]$, $f_n=0$ on $[n+1,2n]$. Clearly, $f_n$ converges uniformly
on compacts to $f$, but the lengths of the longest interval components
of $\{x \in \mathbb{R} : f_n(x)<1 \}$ converge to $\infty$.

However, we will see that the random functions we work with do not
belong to the set of ``problematic'' counter-examples that can arise
and that it will be possible to use Theorem \ref{thm:stablefragfonc}
to get Theorem \ref{thm:stablefrag} and Corollary
\ref{corostablerank}.  Preliminary work will be done in Section
\ref{sec:technicalbackground}, where an explicit construction of the
limit function $H_{\infty}$ via Poisson point measures is also
given. Theorem \ref{thm:stablefrag} and Corollary \ref{corostablerank}
are proved in Section \ref{sec:proof}. Then we will see in Section
\ref{seclastfrag} that the limit $\zeta_{*^{\alpha}}$ arising in
(\ref{eqn:laststable}) is actually distributed as the length to the
power $\alpha$ of an excursion of $H$, conditioned to have its maximum
equal to $1$. It will then be easy to check that this is a size-biased
version of $\zeta$ as defined in (\ref{eqn:biasedzeta}). The bounds
for the tails $\Prob{\zeta_{*^{\alpha}} \geq t}$ will also be
proved in Section \ref{seclastfrag}, as well as the following remark.

\textbf{Remark.} In the Brownian case ($\alpha=-1/2$), the distribution of $\zeta$ (and consequently that of $\zeta_{*^{\alpha}}$), is known; see Section \ref{sec:Brownian}. We do not know the distribution of  $\zeta_{*^{\alpha}}$ explicitly when $\alpha \in (-1/2,0)$. However, in this case, if we set, for $\lambda \geq 0$, 
\[
\Phi(\lambda):= \E{\exp{(-\lambda \zeta_{*^{\alpha}}^{1/\alpha})}}
=\frac{\E{\exp(-\lambda \zeta^{1/\alpha})\zeta^{{-1/\alpha}-1}}}
      {\E{\zeta^{-1/\alpha-1}}},
\]
then it can be shown that $\Phi$ satisfies the following equation:
\begin{equation}
\label{eqPhi}
\begin{split}
\Phi(\lambda)=\exp\!\left(\!-\!\int_{\mathbb{R}^+\times [0,\lambda^{-\alpha}]} \!\!\!
  \left(1-e^{-(-\frac{\alpha}{\alpha+1})^{-1/\alpha}r 
\int_0^t (1-\Phi( v^{-1/\alpha})) 
v^{1/\alpha} \mathrm d v}\right) \right.\\
\left. \times \frac{-\alpha}
{(1+\alpha)^2 \Gamma\left(\frac{1+2\alpha}{1+\alpha}\right)} 
e^{-r(-\frac{\alpha t}{\alpha+1})^{1 + 1/\alpha}} r^{-1/(\alpha+1)}
\mathrm d r \mathrm dt
\right).
\end{split}
\end{equation}

\bigskip

Finally, we recall that the almost sure logarithmic results for $F_1$
and $F_*$ will be proved in Section~\ref{SectionLog}.

%%%%%%%%%%%%%%%%%%%%%%%%
% Technical Background %
%%%%%%%%%%%%%%%%%%%%%%%%

\section{Technical background}
\label{sec:technicalbackground}

We start by detailing the topology on open sets which will give a
proper meaning to the statement of Theorem~\ref{thm:stablefrag}.  We
then recall some facts about height processes and prove various useful
lemmas.  Finally, we introduce the decomposition result of Abraham and
Delmas \cite{Abraham/Delmas}, in a form suitable for our purposes.

\subsection{Topological details}
\label{subsec:topology}

When dealing with interval fragmentations, we will work in the set
$\mathcal{O}$ of bounded open subsets of $\R$.  This is endowed with
the following distance:
\begin{equation*}
d_{\mathcal O}(A,B)=\sum_{k \in \N} 2^{-k} 
d_{\mathcal{H}}\Big((A \cap (-k,k))^{\mathrm{c}} \cap [-k,k],(B
  \cap (-k,k))^{\mathrm{c}} \cap [-k,k]\Big),
\end{equation*}
where $S^{\mathrm{c}}$ denotes the closed complement of $S \in
\mathcal{O}$ and $d_{\mathcal H}$ is the Hausdorff distance on the set
of compact sets of $\R$.  For $A \neq \R$, let $\chi_A(x) = \inf_{y
  \in A^{\mathrm{c}}} |x-y|$.  If we define, for $x \in [-k,k]$,
$\chi^k_A(x) = \chi_{A \cap (-k,k)} (x) = \inf_{y \in (A \cap
  (-k,k))^{\mathrm{c}}} |x-y|$ then we also have
\[
d_{\mathcal{O}}(A,B) = \sum_{k \in \N} 2^{-k} \sup_{x \in
    [-k,k]} |\chi_A^k(x) - \chi_B^k(x)| 
\]
(see p.69 of Bertoin~\cite{BertoinBook}).  The open sets we will deal
with in this paper arise as excursion intervals of continuous
functions.  In particular, we will need to know that if we have a
sequence of continuous functions converging (in an appropriate sense)
to a limit, then the corresponding open sets converge.

Consider the space $C(\R, \R^+)$ of continuous functions from $\R$ to
$\R^+$.  By \emph{uniform convergence on compacts}, we mean
convergence in the metric
\[
d(f,g) = \sum_{k \in \N} 2^{-k} \left( \sup_{t \in [-k,k]} |f(t) -
  g(t)| \wedge 1 \right).
\]
The name is justified by the fact that convergence in $d$ is
equivalent to uniform convergence on all compact sets.

Suppose $f \in C(\R,\R^+)$.  We say that $a \in \R^+$ is a \emph{local
  maximum} of $f$ if there exist $s \in \R$ and $\epsilon > 0$ such
that $f(s) = a$ and $\max_{s-\epsilon \leq t \leq s+\epsilon} f(t) =
a$.  Note that this includes the case where $f$ is constant and equal
to $a$ on some interval, even if $f$ never takes values smaller than
$a$.  We define a \emph{local minimum} analogously.

\begin{proposition} \label{prop:fnconvimpliessetconv} Let $f_n: \R \to
  \R^+$ be a sequence of continuous functions and let $f: \R \to \R^+$
  be a continuous function such that $f(0) = 0$, $f(x) \to \infty$ as
  $x \to +\infty$ or $x \to -\infty$.  Suppose also that $1$ is not a
  local maximum of $f$ and that $\mathrm{Leb}\{x \in \R: f(x) = 1\} =
  0$.  Define $A = \{x \in \R: f(x) < 1\}$ and $A_n = \{x \in \R:
  f_n(x) < 1\}$.  Suppose now that $f_n$ converges to $f$ uniformly on
  compact subsets of $\R$.  Then $d_{\mathcal{O}}(A_n,A) \to 0$ as $n
  \to \infty$.
\end{proposition}

Define
\[
g(x) = \sup\{y \leq x: f(y) \geq 1\}
\]
and
\[
d(x) = \inf\{y \geq x: f(y) \geq 1\}.
\]
Then 
\begin{equation} \label{eqn:chiequiv}
\chi_A(x) = (x - g(x)) \wedge (d(x) - x).
\end{equation}
Define $g_n$ and $d_n$ to be the analogous quantities for $f_n$.  The
proof of the following lemma is straightforward.

\begin{lemma} \label{lem:chi}
For $x, y \in \R$, 
\[
|\chi_{A}(x) - \chi_{A}(y)| \leq |x-y|.
\]
Moreover,
\[
|\chi_{A_n}(x) - \chi_A(x)| \leq \max \{|g_n(x) - g(x)|,|d_n(x) - d(x)|\}.
\]
\end{lemma}

\begin{proof}[Proof of Proposition~\ref{prop:fnconvimpliessetconv}.]
  It suffices to prove that $d_{\mathcal{H}}((A_n \cap
  (-1,1))^{\mathrm{c}} \cap [-1,1],(A \cap (-1,1))^{\mathrm{c}} \cap
  [-1,1]) \to 0$ as $n \to \infty$ for $f,f_n : [-1,1] \to \R^+$ such
  that $f(0) = 0$, 1 is not a local maximum of $f$, $\mathrm{Leb}\{x
  \in [-1,1]: f(x) = 1\} = 0$ and $f_n \to f$ uniformly.  In other
  words, we need to show that $\sup_{x \in [-1,1]}| \chi^1_{A_n}(x) -
  \chi^1_{A}(x)| \to 0$ as $n \to \infty$.  Note that the appropriate
  definitions of $g(x)$ and $d(x)$ in order to make
  (\ref{eqn:chiequiv}) true for $\chi^1_{A}$ are
\[
g(x) = \sup\{y \leq x: f(y) \geq 1\} \vee -1 
\]
and
\[
d(x) = \inf\{y \geq x: f(y) \geq 1\} \wedge 1
\]
and we adopt these definitions for the rest of the proof.

Let $\epsilon > 0$.  For $r>0$ let
\begin{align*}
E^{\uparrow}_r & = \{x \in (-1,1): x \in (a,b) \text{ such that } f(y) > 1,
\forall\ y \in (a,b), |b-a| > r\}, \\
E^{\downarrow}_r & = \{x \in (-1,1): x \in (a,b) \text{ such
  that } f(y) < 1, \forall\ y \in (a,b), |b-a| > r\}.
\end{align*}
These are the collections of excursion intervals of length exceeding
$r$ above and below 1.  Take $0 < \delta < \epsilon/2$ small enough that
$\mathrm{Leb}(E^{\uparrow}_{\delta} \cup E^{\downarrow}_{\delta}) > 2
- \epsilon/2$ (we can do this since $\mathrm{Leb}\{x \in [-1,1]: f(x)
= 1\} = 0$).  Set $R = [-1,1] \setminus (E^{\uparrow}_{\delta} \cup
E^{\downarrow}_{\delta})$.  Each of $E^{\uparrow}_{\delta}$ and
$E^{\downarrow}_{\delta}$ is composed of a finite number of open
intervals.

$\bullet$ We will first deal with $E^{\uparrow}_{\delta}$.  On this
set, $\chi_A^1(x) = 0$.  Take hereafter $0 < \eta < \delta/6$ and let
\[
E^{\uparrow}_{\delta,\eta} = \{x \in E^{\uparrow}_{\delta}: (x-\eta,x+\eta)
\subset E^{\uparrow}_{\delta}\}.
\]
Then $\inf_{x \in E_{\delta,\eta}^{\uparrow}} f(x) > 1$.  Since $f_n
\to f$ uniformly, it follows that there exists $n_1$ such that $f_n(x)
> 1$ for all $x \in E^{\uparrow}_{\delta,\eta}$ whenever $n \geq n_1$.
Then for $n \geq n_1$, we have $\chi_{A_n}^1(x) = 0$ for all $x \in
E^{\uparrow}_{\delta,\eta}$.  Since $|\chi_{A_n}^1(x) - \chi_{A_n}^1(y)|
\leq |x-y|$, it follows that $\chi_{A_n}^1(x) < \eta$ for all $x \in
E^{\uparrow}_{\delta}$.  So $\sup_{x \in E^{\uparrow}_{\delta}}
|\chi_A^1(x) - \chi_{A_n}^1(x)| < \eta$ whenever $n \geq n_1$.

$\bullet$ Now turn to $E^{\downarrow}_{\delta}$.  As before, define
\[
E^{\downarrow}_{\delta,\eta} = \{x \in E^{\downarrow}_{\delta}:
(x-\eta,x+\eta) \subset E^{\downarrow}_{\delta}\}.
\]
Then $\sup_{x \in E^{\downarrow}_{\delta,\eta}} f(x) < 1$.  Since $f_n
\to f$ uniformly, it follows that there exists $n_2$ such that $f_n(x)
< 1$ for all $x \in E^{\downarrow}_{\delta,\eta}$ whenever $n \geq
n_2$.  Now, for each excursion below 1, there exists a left end-point
$g$ and a right end-point $d$.  For all $x$ in the same excursion,
$g(x) = g$ and $d(x) = d$.  Suppose first that we have $g \neq -1$, $d
\neq 1$ (in this case we say that the excursion \emph{does not touch
  the boundary}).  Since 1 is not a local maximum of $f$, there must
exist $z_g < g$ and $z_d > d$ such that $|z_g-g| < \eta$, $|z_d - d| <
\eta$, $f(z_g)> 1$ and $f(z_d) > 1$.  

Suppose there are $N_{\delta}$ excursions below 1 of length greater
than $\delta$ which do not touch the boundary.  To excursion $i$ there
corresponds a left end-point $g_i$, a right end-point $d_i$
and points $z_{g_i}$, $z_{d_i}$, $1 \leq i \leq N_{\delta}$.
Write 
\[
\tilde{E}^{\downarrow}_{\delta} = \cup_{i=1}^{N_{\delta}}
(g_i,d_i) \text{ and }
\tilde{E}^{\downarrow}_{\delta,\eta} =
\tilde{E}^{\downarrow}_{\delta} \cap E^{\downarrow}_{\delta,\eta} =
\cup_{i=1}^{N_{\delta}} (g_i+\eta,d_i-\eta).
\]
Then $\min_{1 \leq i \leq N_{\delta}} (f(z_{g_i}) \wedge f(z_{d_i})) >
1$.  Since $f_n \to f$ uniformly, there exists $n_3$ such that
$\min_{1 \leq i \leq N_{\delta}} (f_n(z_{g_i}) \wedge f_n(z_{d_i})) >
1$ for all $n \geq n_3$.  So for $n \geq n_2 \vee n_3$ and any $x \in
\tilde{E}^{\downarrow}_{\delta,\eta}$, by the intermediate value
theorem, there exists at least one point $a_n(x) \in (g(x) - \eta,
g(x) + \eta)$ such that $f_n(a_n(x)) = 1$ and at least one point
$b_n(x) \in (d(x) - \eta, d(x) + \eta)$ such that $f_n(b_n(x)) = 1$.
Since $g(x)$ and $d(x)$ are constant on excursion intervals, it
follows that $\sup_{x \in \tilde{E}^{\downarrow}_{\delta,\eta}}
|g_n(x) - g(x)| < \eta$ and $\sup_{x \in
  \tilde{E}^{\downarrow}_{\delta,\eta}} |d_n(x) - d(x)| < \eta$ for $n
\geq n_2 \vee n_3$.  Hence, by Lemma~\ref{lem:chi},
\[
\sup_{x \in \tilde{E}^{\downarrow}_{\delta,\eta}} |\chi_A^1(x) -
\chi_{A_n}^1(x)| < \eta
\]
whenever $n \geq n_2 \vee n_3$.  Since $|\chi_A^1(x) - \chi_A^1(y)|
\leq |x-y|$ and $|\chi_{A_n}^1(x) - \chi_{A_n}^1(y)| \leq |x-y|$, by
using the triangle inequality we obtain that
\[
\sup_{x \in \tilde{E}^{\downarrow}_{\delta}} |\chi_A^1(x) -
  \chi_{A_n}^1(x)| < 3 \eta.
\]

It remains to deal with any excursions in $E^{\downarrow}_{\delta}$
which touch the boundary.  Clearly, there is at most one excursion in
$E^{\downarrow}_{\delta}$ touching the left boundary and at most one
excursion touching the right boundary.  For these excursions, we can
argue as before at the non-boundary end-points.  At the boundary
end-points, the argument is, in fact, easier since we have (by
construction) $\chi_{A}^1(-1) = \chi_{A_n}^1(-1) = 0$ and
$\chi_{A}^1(1) = \chi_{A_n}^1(1) = 0$.  So, there exists $n_4$ such
that for all $n \geq n_2 \vee n_3 \vee n_4$,
\[
\sup_{x \in E^{\downarrow}_{\delta}} |\chi_A^1(x) - \chi_{A_n}^1(x)| <
3 \eta.
\]

$\bullet$ For any $x \in R$, we have $\chi_A^1(x) \leq \delta/2$.
Moreover, since $\mathrm{Leb}(E^{\uparrow}_{\delta} \cup
E^{\downarrow}_{\delta}) > 2 - \epsilon/2$, there must exist a point
$z(x) \in R$ such that $|z(x)- x| < \epsilon/2$ which is the end-point
of an excursion interval (above or below 1) of length exceeding
$\delta$. So for all $x \in R$ and all $n$ we have
\begin{align*}
|\chi_A^1(x) - \chi_{A_n}^1(x)| & \leq |\chi_A^1(x)| + |\chi_{A_n}^1(x) -
\chi_{A_n}^1(z(x))| + |\chi_{A_n}^1(z(x)) - \chi_{A}^1(z(x))| \\
& \leq \delta/2 + |x-z(x)| + \sup_{y \in E^{\uparrow}_{\delta} \cup
  E^{\downarrow}_{\delta}} | \chi_{A_n}^1(y) - \chi_A^1(y)| \\
& < \delta/2 + \epsilon/2 + \sup_{y \in E^{\uparrow}_{\delta} \cup
  E^{\downarrow}_{\delta}} | \chi_{A_n}^1(y) - \chi_A^1(y)|
\end{align*}
(note that at the second inequality we use the continuity of $\chi_A^1$
and $\chi_{A_n}^1$ and the fact that $\chi^1_A(z(x)) = 0$).

$\bullet$ Finally, let $n_0 = n_1 \vee n_2 \vee n_3 \vee n_4$.  Then
since $\eta < \delta/6$ and $\delta < \epsilon/2$, for $n \geq n_0$ we
have
\[
\sup_{x \in [-1,1]} |\chi_A^1(x) - \chi_{A_n}^1(x)| < \epsilon.
\]
The result follows.
\end{proof}

The following lemma will be used implicitly in Section \ref{sec:convheightprocess}.

\begin{lemma} \label{lem:babyanalysis} Suppose that $a > 0$, $\alpha
  \in \R$ and that $f: \R \to \R^+$ is a continuous function.  Let
  $g(t) = a^{\alpha} f(at)$.  Then the function $(a,f) \mapsto g$ is
  continuous for the topology of uniform convergence on compacts.
\end{lemma}

\begin{proof}
Let $f_n: \R \to \R^+$ be a
sequence of continuous functions with $f_n \to f$ uniformly on
compacts.  Suppose that $a_n$ is a sequence of reals with $a_n \to a >
0$.  Suppose $K \subset \R$ is a compact set.  Then we have
\begin{align*}
& \sup_{t \in K} | a^{\alpha} f(at) - a_n^{\alpha} f_n(a_n t)| \\
& \leq \sup_{t \in K} | a^{\alpha} f(at) - a_n^{\alpha} f(at)| 
+ \sup_{t \in K} |a_n^{\alpha} f(at) - a_n^{\alpha} f(a_n t)|
+ \sup_{t \in K} |a_n^{\alpha} f(a_n t) - a_n^{\alpha} f_n(a_n t)| \\
& \leq \sup_{t \in aK} |f(t)| |a^{\alpha} - a_n^{\alpha}|
+ a_n^{\alpha}  \sup_{t \in K} |f(at) - f(a_n t)|
+ a_n^{\alpha}  \sup_{t \in K} |f(a_n t) - f_n(a_n t)|.
\end{align*}
The set $aK$ is compact and so $f$ is bounded on it; it follows that
the first term converges to $0$.  Take $0 < \epsilon < a$.  Since $a_n
\to a$, there exists $n$ sufficiently large that $|a_n - a| <
\epsilon$.  Define $\tilde{K} = \{bt: t \in K, b \in [a - \epsilon,a +
\epsilon]\}$.  Then $\tilde{K}$ is also a compact set.  The second
term converges to $0$ because $f$ is uniformly continuous on
$\tilde{K}$.  The third term is bounded above by
$((a-\epsilon)^{\alpha} \vee (a+\epsilon)^{\alpha}) \sup_{t \in
  \tilde{K}} | f(t) - f_n(t)|$ and so, since $f_n \to f$ uniformly on
compacts, this converges to $0$.
\end{proof}

Finally, we will need the following lemma in Section \ref{sec:proof}.

\begin{lemma} \label{lem:cvLeb}
  Let $f: \mathbb{R} \rightarrow \mathbb{R}^+$ be a continuous
  function such that
\[
\mathrm{Leb}\{x \in \R : f(x)=1\}=0.
\] 
Suppose $f_n: \mathbb{R} \rightarrow \mathbb{R}^+$ is a sequence of
continuous functions that converges to $f$ uniformly on compacts.
Then, for all $K > 0$, as $n \rightarrow \infty$,
\[
\mathrm{Leb}\{x \in [-K,K] : f_n(x)<1\} \rightarrow \mathrm{Leb}\{x \in
[-K,K] : f(x)<1\}.
\]
\end{lemma}

\begin{proof}
Let $K>0$ and fix $\varepsilon>0$. For all $n$ sufficiently large
\[
f(x)-\varepsilon \leq f_n(x) \leq f(x) + \varepsilon, \text{ for all
  $x \in [-K,K]$},
\]
hence
\begin{align*}
\mathrm{Leb}\{x \in [-K,K] : f(x)<1-\varepsilon\} 
& \leq \mathrm{Leb}\{x \in [-K,K] : f_n(x)<1\} \\ 
&\leq \mathrm{Leb}\{x \in [-K,K] : f(x)<1+\varepsilon\}. 
\end{align*}
When $\varepsilon \to 0$, the left-hand side of this inequality
converges to $\mathrm{Leb}\{x \in [-K,K] : f(x)<1\}$ and the right-hand side to
$\mathrm{Leb}\{x \in [-K,K] : f(x) \leq1\} $, which are equal by assumption.
\end{proof}

%%%%%%%%%%%%%%%%%%%%
% Height processes %
%%%%%%%%%%%%%%%%%%%%

\subsection{Height processes} \label{subsec:heightprocesses}

Here, we define the stable height process and recall some of its
properties.  We refer to Le Gall and Le Jan \cite{LeGall/LeJan} and
Duquesne and Le Gall \cite{Duquesne/LeGall} for background.  (All of
the definitions and results stated without proof below may be found in
these references.)

Suppose that $X$ is a spectrally positive stable L\'evy process with
Laplace exponent $\lambda^{\beta}$, $\beta \in (1,2]$.  That is,
$\E{\exp(-\lambda X_t)}=\exp(t\lambda^{\beta})$ for all $\lambda,t
\geq 0$ and, therefore, if $\beta \in (1,2)$, the L\'evy measure of
$X$ is $\beta(\beta-1)(\Gamma(2-\beta))^{-1} x^{-\beta-1}\mathrm d x$,
$x>0$.  Let $I(t):=\inf_{0 \leq s \leq t} X(s)$ be the infimum process
of $X$.  For each $t>0$, consider the time-reversed process defined
for $0 \leq s <t$ by
\[
\hat{X}^{(t)}(s):=X(t)-X((t-s)-),
\]
and let $(\hat{S}^{(t)}(s),0 \leq s \leq t)$ be its supremum, that is
$\hat{S}^{(t)}(s) = \sup_{u \leq s} \hat{X}^{(t)}(u)$.  Then 
the height process $H(t)$ is defined to be the local time at $0$ of
$\hat{S}^{(t)}-\hat{X}^{(t)}$.

It can be shown that the process $H$ possesses a continuous version
(Theorem 1.4.3 of \cite{Duquesne/LeGall}), which we will implicitly
consider in the following.  The scaling property of $X$ is inherited
by $H$ (see Section 3.3 of \cite{Duquesne/LeGall}) as follows: for all
$a > 0$,
\begin{equation*}
  (a^{1/{\beta}-1}H(ax),x\geq 0) \equidist (H(x), x \geq 0).
\end{equation*}
When $\beta=2$, the height process is, up to a scaling factor, a
reflected Brownian motion.

The excursion measure of $X-I$ away from $0$ is denoted by $\N$. \emph{In
the following, we work under this excursion measure}.  Let
$\mathcal{E}$ be the space of excursions; that is, continuous
functions $f:\R^+ \to \R^+$ such that $f(0) = 0$, $\inf\{t > 0: f(t) =
0\} < \infty$ and if $f(s) = 0$ for some $s > 0$ then $f(t) = 0$ for
all $t > s$.  The lifetime of $H \in \mathcal{E}$ is then denoted by
$\sigma$, that is
\[
\sigma := \inf\{x > 0: H(x) = 0\}.
\] 
We define its maximum to be
\[
H_{\max}:=\max_{x \in [0,\sigma]} H(x).
\]
We will also deal with (regular versions of) the probability measures
$\N(\cdot| \sigma = v)$, $v > 0$ and $\N(\cdot | H_{\max} = m)$, $m >
0$.   The following proposition is implicit in Section 3 of Abraham
and Delmas~\cite{Abraham/Delmas} and is also a consequence of
Theorem~\ref{thm:largestfrag} (ii) below.

\begin{proposition} 
  For any $v > 0$, under $\N(\cdot| \sigma = v)$ there exists an almost
  surely unique point $x_{\max}$ at which $H$ attains its maximum,
  that is
\[
x_{\max}:=\inf \{x \in [0,\sigma]:H(x)=H_{\max} \}.
\]
\end{proposition}

Note that $\mathbf{e}$, $\zeta$ and $x_*$ (see Section
\ref{sec:stable}) have the distributions of $H$, $H_{\max}$ and
$x_{\max}$ respectively under $\N( \cdot |\sigma=1)$.

First we note the tails of certain measures.
\begin{proposition} \label{prop:tails}
For all $m>0$,
\begin{gather} 
\mathbb N\left( H_{\max}>m\right)=(\beta-1)^{1/(1-\beta)}m^{\frac{1}{1-\beta}},
\label{eqn:tail1} \\
\mathbb N\left(\sigma>m\right)=(\Gamma(1-1/\beta))^{-1}m^{-\frac{1}{\beta}}. 
\label{eqn:tail2}
\end{gather}
\end{proposition}

\begin{proof}
  For (\ref{eqn:tail1}) see, for example, Corollary 1.4.2 and Section
  2.7 of Duquesne and Le Gall~\cite{Duquesne/LeGall}.  It is well
  known (Theorem 1, Section VII.1 of \cite{BertoinLevy}) that the
  right inverse process $J = (J(t), t \geq 0)$ of $I$ defined by
  $J(t):=\inf\{u \geq 0: -I(u) > t\}$ is a stable subordinator with a
  L\'evy measure $(\beta\Gamma(1-1/\beta))^{-1}x^{-1-1/{\beta}}
  \mathrm d x$, $x>0$.  Since $\N(\sigma \in \mathrm{d}m)$ is equal to
  this L\'evy measure, (\ref{eqn:tail2}) follows.
\end{proof}

Recall that $\alpha = 1/\beta - 1$.  We will, henceforth, primarily
work in terms of $\alpha$ rather than $\beta$.  We will make extensive
use of the scaling properties of the height function under the
excursion measure.  For $m>0$, let $H^{[m]}(x) = m^{-\alpha}H(x/m)$.
Note that if $H$ has lifetime $\sigma$ then $H^{[m]}$ has lifetime $m
\sigma$ and maximum height $m^{-\alpha} H_{\max}$.  Note also that
$(H^{[m]})^{[a]} = H^{[ma]}$, for all $m,a>0$. The following
proposition, which summarizes results from Section 3.3 of Duquesne and Le Gall~\cite{Duquesne/LeGall}, gives a version of the
scaling property for the height process conditioned on its lifetime.

\begin{proposition} \label{prop:Miermont}
For any test function $f: \mathcal{E} \to \R$ and any $x,m > 0$, we have
\[
\N[f(H^{[m]}) | \sigma = x/m] = \N[f(H)| \sigma = x].
\]
Moreover, for any $\eta > 0$,
\[
\N[f(H) | \sigma = x] = \N[f(H^{[x/\sigma]}) | \sigma > \eta].
\]
\end{proposition}

We now state two lemmas that will play an essential role in the proof
of Theorem \ref{thm:stablefragfonc}. The first lemma gives the scaling
property for $H$ conditioned on its maximum.

\begin{lemma} \label{lem:scaling} Let $f: \mathcal{E} \times \R^+
  \times \R^+ \to \R$ be any test function.  For all $m > 0$,
\[
\N[f(H^{[m]}, m \sigma, m^{-\alpha} H_{\max})] = m^{1 + \alpha}
\N[f(H, \sigma, H_{\max})].
\]
Moreover, for all $x, a > 0$,
\[
\N[f(H^{[x/\sigma]}, \sigma, H_{\max})] = a^{-1-\alpha}
\N[f(H^{[x/\sigma]}, a \sigma, a^{-\alpha} H_{\max})].
\]
In particular, for any test function $g: \mathcal{E} \times \R^+ \to
\R$,
\[
\N[g(H^{[x]}, \sigma) | H_{\max} = u] 
= \N[g(H, x^{-1} \sigma) | H_{\max} = x^{-\alpha} u]
\]
and
\[
\N[g(H^{[x/\sigma]}, \sigma) | H_{\max} = u] 
= \N[g(H^{[x/\sigma]}, x^{-1} \sigma) | H_{\max} = x^{-\alpha} u].
\]
\end{lemma}

\begin{proof}
By conditioning on the value of $\sigma$ and the tails in
Proposition~\ref{prop:tails}, we have
\[
\N[f(H^{[m]}, m \sigma, m^{-\alpha} H_{\max})]
= c \int_{\R^+} \N[f(H^{[m]}, mb, m^{-\alpha} H_{\max}) | \sigma = b]
b^{-\alpha-2} \mathrm{d} b,
\]
for some constant $c$.  By Proposition \ref{prop:Miermont},
\[
 \int_{\R^+} \N[f(H^{[m]}, mb, m^{-\alpha} H_{\max}) | \sigma = b]
b^{-\alpha-2} \mathrm{d} b 
= \int_{\R^+} \N[f(H,mb,H_{\max}) | \sigma = mb]
b^{-\alpha-2} \mathrm{d} b.
\]
Changing variable with $a = mb$ gives
\[
c m^{1 + \alpha} \int_{\R^+} \N[f(H,a,H_{\max}) | \sigma = a]
a^{-\alpha-2} \mathrm{d} a = m^{1+\alpha} \N[f(H, \sigma, H_{\max})],
\]
which yields the first statement.  The second statement is a
consequence of the first and the conditioned statements follow easily.
\end{proof}

Finally, we relate the law of $H$ conditioned on its lifetime and the
law of $H$ conditioned on its maximum.  For the rest of the paper, $c$
denotes a positive finite constant that may vary from line to line.

\begin{lemma} \label{lem:sigmaH_max} For all test
  functions $f: \mathcal{E} \to \R$ and all $x > 0$,
\[
\N[f(H) | \sigma = x] = \Gamma(-\alpha) \left(\frac{-\alpha}{ \alpha + 1}\right)^{1/\alpha} \int_{\R^{+}} \N[f(H^{[x/\sigma]}) \I{\sigma
  > x} | H_{\max} = x^{-\alpha} u]
u^{1/\alpha} \mathrm{d}u.
\]
Moreover, for all non-negative test functions $g: \R^+ \to \R$,
\[
\N[g(\sigma^{\alpha}) | H_{\max} = 1] 
= \frac{\N[g(H_{\max}) H_{\max}^{-1/\alpha - 1} | \sigma = 1]}
       {\N[H_{\max}^{-1/\alpha - 1} | \sigma = 1]}.
\]
\end{lemma}

\begin{proof}
  Taking $\eta = 1$ in the second statement of Proposition
  \ref{prop:Miermont}, we see that $\N[f(H) | \sigma = x]$ is
  equal to $\N[f(H^{[x/\sigma]}) \I{\sigma > 1}]/\mathbb N(\sigma>1)$. Then,
  conditioning on the value of $H_{\max}$ and using (\ref{eqn:tail1}), we have
\[
\N[f(H^{[x/\sigma]}) \I{\sigma > 1}] = \left(\frac{-\alpha}{\alpha+1}\right)^{1/\alpha} \int_{\R^{+}}
\N[f(H^{[x/\sigma]}) \I{\sigma > 1} | H_{\max} = u]
u^{1/\alpha} \mathrm{d}u.
\]
By the final statement of Lemma~\ref{lem:scaling}, we see that this is
equal to
\[
\left(\frac{-\alpha}{\alpha+1}\right)^{1/\alpha} \int_{\R^{+}}
\N[f(H^{[x/\sigma]}) \I{\sigma > x} | H_{\max} = x^{-\alpha} u]
u^{1/\alpha} \mathrm{d}u.
\]
The first statement follows by noting from (\ref{eqn:tail2}) that $\N(\sigma > 1) = \Gamma(-\alpha)^{-1}$.

In order to prove the second statement in the lemma, note that by the
first statement we have
\[
\N[g(H_{\max}) H_{\max}^{-1/\alpha - 1} | \sigma = 1]
= c \int_{\R^+} \N[g(\sigma^{\alpha} H_{\max}) \sigma^{-1-\alpha}
H_{\max}^{-1/\alpha - 1} \I{\sigma > 1} |
H_{\max} = u] u^{1/\alpha} \mathrm{d} u.
\]
By Lemma~\ref{lem:scaling}, we have that
\[
\N[g(\sigma^{\alpha} H_{\max}) \I{\sigma > 1} \sigma^{-1-\alpha}
H_{\max}^{-1/\alpha - 1} | H_{\max} = u] 
= \N[g(\sigma^{\alpha}) \sigma^{-1-\alpha} \I{\sigma > u^{1/\alpha}} |
H_{\max} = 1],
\]
for all $u > 0$.  Hence, by Fubini's theorem,
\begin{align*}
\N[g(H_{\max}) H_{\max}^{-1/\alpha - 1} | \sigma = 1]
& = c \N\left[g(\sigma^{\alpha}) \sigma^{-1-\alpha}
\int_{\sigma^{\alpha}}^{\infty} u^{1/\alpha} \mathrm{d}u \bigg|
H_{\max} = 1\right] \\
& = c \N[g(\sigma^{\alpha}) | H_{\max} = 1].
\end{align*}
The result follows.
\end{proof}

\subsection{Williams' decomposition} 

\label{sec:Williams}

Except in the Brownian case, the height process is not Markov.
However, a version of it can be reconstructed from a measure-valued
Markov process $\rho$, called the \emph{exploration process} (see Le
Gall and Le Jan \cite{LeGall/LeJan} or Section 0.3 of Duquesne and Le
Gall \cite{Duquesne/LeGall} for a definition), by taking $H(t)$ to be
the supremum of the topological support of $\rho(t)$. Abraham and
Delmas \cite{Abraham/Delmas} give a decomposition of the height
process $H$ (more precisely, of the continuum random tree coded by this
height process) in terms of this exploration process.  This
decomposition is the analogue of Williams' decomposition for the
Brownian excursion discussed earlier.  We recall their result below
but we state it in terms of the height process.  This is somewhat less
precise, but is sufficient for our purposes and easier to state.

\noindent \textbf{Notation.}  Take an arbitrary
 point measure $\mu=\sum_{i\geq 1}a_i\delta_{t_i}$ on
$(0,\infty)$.  Now consider, for each $i \geq 1$, an independent
Poisson point process on $\R^+ \times \mathcal{E}$ of intensity
$\mathrm{d}u \N( \cdot ,H_{\max}\leq t_i )$, having points
$\{(u_j^{(i)},f_j^{(i)}),j \geq 1\}$.  For each $i \geq 1$, define a
subordinator $\tau^{(i)}$ by
\[
\tau^{(i)}(u)=\sum_{j: u_j^{(i)} \leq u} \sigma(f_j^{(i)}), \quad u \geq 0,
\]
where for any excursion $f$, $\sigma(f)$ denotes its length.  Note
that the L\'evy measure of this subordinator integrates the function
$1 \wedge x$ on $\mathbb R^+$ since $\mathbb N[1\wedge \sigma,
H_{\text{max}}\leq t_i] \leq \mathbb N[1\wedge \sigma]$, which is finite by
Proposition~\ref{prop:tails}. Hence $\tau^{(i)}(u)<\infty$ for all $u
\geq 0$ a.s.

We will need the function $H^{(i)}$, defined on $[0,\tau^{(i)}(a_i)]$
by
\begin{equation}
\label{eqn:Hdefinition}
H^{(i)}(x) := \sum_{j: u_j^{(i)} \leq a_i} f_j^{(i)}(x-\tau^{(i)}(u_j^{(i)}-))
\I{\tau^{(i)}(u_j^{(i)}-) < x \leq \tau^{(i)}(u_j^{(i)})}.
\end{equation}
The process $H^{(i)}$ can be viewed as a collection of excursions of
$H$ conditioned to have heights lower than $t_i$ and such that the local
time of $H^{(i)}$ at $0$ is $a_i$. 

We will now use this set-up in the situation of interest.  Let
$\rho:=\sum_{i \geq 1} \delta_{(v_i,r_i,t_i)}$ be a Poisson point
measure on $[0,1]\times \mathbb R^+ \backslash \{0\} \times \mathbb
R^+ \backslash \{0\}$ with intensity
\begin{equation*}
\frac{\beta(\beta-1)}{\Gamma(2-\beta)}
\exp(-rc_{\beta}t^{1/(1-\beta)})r^{-\beta}\mathrm d v\mathrm d r
\mathrm dt,
\end{equation*}
where $c_{\beta}=(\beta-1)^{1/(1-\beta)}$.  Conditionally on the
 point measures $\sum_{i \geq 1}(1-v_i)r_i \delta_{t_i}$ and $\sum_{i
   \geq 1}v_ir_i \delta_{t_i} $, use them to define two independent collections of
 independent processes $\{H_+^{(i)}, i \geq 1\}$ and $\{H_{-}^{(i)}, i
 \geq 1\}$ respectively, as in (\ref{eqn:Hdefinition}).  We now glue
 these together in order to define a function $H_{\infty}$ on $\R$.
 For $u \geq 0$, set
\[
\eta^{+}(u):=\sum_{i: t_i \leq u}\tau^{(i)}_+((1-v_i)r_i), \quad 
\eta^{-}(u):=\sum_{i: t_i \leq u}\tau^{(i)}_-(v_ir_i).
\]
It is not obvious that $\eta^+(u)$ and $\eta^-(u)$ are almost surely
finite for all $u\geq0$. This is a consequence of the forthcoming
Theorem \ref{thm:CVHcenter}. It can also be proved via Campbell's
theorem for Poisson point processes.  Now set
\begin{align*}
H_{\infty}(x) := & \left(\sum_{i \geq 1} 
\left[t_i-H_{+}^{(i)}(x-\eta^{+}(t_i-))\right] 
\I{\eta^{+}(t_i-) < x \leq \eta^{+}(t_i)} \right) \I{x \geq 0}  \\
& + \left( \sum_{i \geq 1} \left[t_i-H_{-}^{(i)}(-x-\eta^{-}(t_i-))\right] 
\I{\eta^{-}(t_i-) < -x \leq \eta^{-}(t_i)} \right) \I{x < 0}.
\end{align*}

\begin{figure}
\begin{center}
\psfrag{m}{$m$}
\psfrag{f}{$\eta^+(m)$}
\psfrag{e}{$-\eta^-(m)$}
\psfrag{a}{$t_l$}
\psfrag{b}{$t_k$}
\psfrag{c}{$t_j$}
\psfrag{d}{$t_i$}
\includegraphics{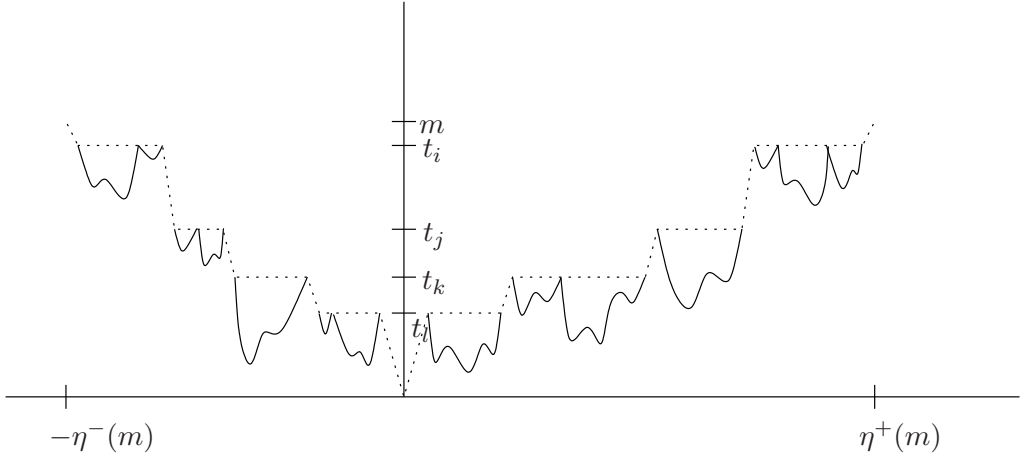}
\end{center}
\caption{Schematic drawing of $H_{\infty}$, with the one-sided running
  suprema indicated by dotted lines.}
\label{fig:Hinfty}
\end{figure}

Note that almost surely for all $u>0$, 
\begin{equation*} \label{etaplus}
\eta^+(u)=\inf\{x>0: H_{\infty}(x)>u\}
\end{equation*}
or, equivalently, the right inverse of $\eta^+$ is equal to the
supremum process 
\[
\left(\sup_{0 \leq y \leq x}H_{\infty}(y), x \geq 0\right).
\]
Symmetrically,
\begin{equation*}\label{etamoins}
\eta^-(u)=\sup\{x<0: H_{\infty}(x)>u\}
\end{equation*} 
and the right inverse of $\eta^-$ is the supremum process 
\[
\left(\sup_{-x \leq y \leq 0} H_{\infty}(y), x \geq 0\right).
\]
Roughly, the construction of $H_{\infty}$ works as follows:
conditional on the supremum (and for each value of the supremum), we
take a collection of independent excursions below that supremum which
are conditioned not to go below the $x$-axis, and which have total
local time $r_i$ for an appropriate value $r_i$.  This local time is
split with a uniform random variable into a proportion $v_i$ which
goes to the left of the origin and a proportion $(1 - v_i)$ which goes
to the right of the origin.  See Figure~\ref{fig:Hinfty} for an idea
of what $H_{\infty}$ looks like (note that the times $t_i$ should be
dense on the vertical axis).  Note that we may always choose a
continuous version of $H_{\infty}$.  Roughly, this is because the
processes $\eta^{+}$ and $\eta^{-}$ almost surely have no intervals where
they are constant.  This implies that the one-sided suprema of
$H_{\infty}$ are continuous.  Finally, the excursions that we glue
beneath these suprema can be assumed to be continuous.

The following theorem says that if we flip this picture over we obtain
an excursion of the height process which is conditioned on its maximum
height.  The proof follows easily from Lemma 3.1 and Theorem 3.2 of Abraham
and Delmas~\cite{Abraham/Delmas}.

\begin{theorem} [Abraham and Delmas \cite{Abraham/Delmas} (Stable
  case, $1<\beta<2$)]
\label{thm:CVHcenter} 
For all $m>0$,
\begin{align*}
& \left(m-H_{\infty}(x-\eta^{-}(m)),0 \leq x \leq
\eta^-(m)+\eta^+(m) \right) \\
& \qquad \equidist \left(H(x),0
\leq x \leq \sigma \right) \text{ under } \mathbb N(\cdot |
H_{\max}=m).
\end{align*}
\end{theorem}

We note that, in particular, \[\eta(m):=\eta^-(m)+\eta^+(m)\] has the
distribution of $\sigma$ under $\N ( \cdot | H_{\max}=m)$ and that
$\eta^{-}(m)$ has the distribution of $x_{\max}$ under the same measure.

Theorem 3.2 of Abraham and Delmas~\cite{Abraham/Delmas} also entails a
Brownian counterpart of this result.  Much of the complication in the
description of $H_{\infty}$ for $1 < \beta < 2$ came from the fact
that the stable height process has repeated local minima.  Here the
construction of $H_{\infty}$ is simpler since local minima are unique
in the Brownian excursion.  Let $\sum_{i \geq 1} \delta_{(t_i,h_i)}$
and $\sum_{i \geq 1} \delta_{(u_i,f_i)}$ be two independent Poisson
point measures on $\R_{+} \times \mathcal{E}$, both with intensity
$\mathrm d t \mathbb N(\cdot, H_{\max}\leq t )$.  For $s \geq 0$, set
\[ 
\eta^{+}(s):= \sum_{i: t_i \leq s}\sigma(h_i), 
\quad \eta^{-}(s):=\sum_{i: u_i \leq s}\sigma(f_i).
\]
Finally, set
\begin{align*}
H_{\infty}(x):= & \left( \sum_{i \geq 1} \left[t_i-h_i(x-\eta^{+}(t_i-))\right]
\I{\eta^+(t_i-) < x \leq \eta^+(t_i)} \right) \I{x \geq 0} \\
& + \left( \sum_{i \geq 1} \left[u_i - f_i(-x - \eta^{-}(u_i-))\right]
    \I{\eta^-(u_i-) < -x \leq \eta^-(u_i)} \right) \I{x < 0}.
\end{align*}

\begin{theorem} [Abraham and Delmas~\cite{Abraham/Delmas},
  Williams~\cite{Rogers/Williams2} (Brownian case, $\beta=2$)]
\label{thm:CVHcenterBr}
For all $m>0$,
\begin{align*}
& \left(m-H_{\infty}(x-\eta^{-}(m)),0 \leq x \leq
\eta^-(m)+\eta^+(m) \right)\\
& \qquad \equidist \left(H(x),0
\leq x \leq \sigma \right) \text{ under } \mathbb N(\cdot |
H_{\max}=m).
\end{align*}
\end{theorem}

In the sequel, we will need various properties of the processes
$(H_{\infty}(x), x \in \R)$ and $(\eta(m), m \geq 0)$.  
We start with some properties of $H_{\infty}$. 

\begin{lemma} \label{lem:propofH_infty}
  For all $\beta \in (1,2]$, the process $H_{\infty}$ is self-similar:
  for all $m \in \R$,
\[
(m^{\alpha}H_{\infty}(mx), x \in \R) \equidist
  (H_{\infty}(x), x \in \R).
\]
Moreover, with probability 1, $H_{\infty}(x) \rightarrow +\infty$ as
$x \rightarrow +\infty$ or $x \rightarrow -\infty$.
\end{lemma}

\begin{proof}
  The self-similarity property is an easy consequence of the identity
  in law stated in Theorems \ref{thm:CVHcenter} and
  \ref{thm:CVHcenterBr} and of the scaling property of $H$ conditioned
  on its maximum (Lemma \ref{lem:scaling}).

  Now, we will show that for each $A>0$, a.s. $H_{\infty}(t)>A$ for
  $t$ sufficiently large.  This will imply that $\liminf_{x\rightarrow
    +\infty} H_{\infty}(x)$ is almost surely larger than $A$ for all
  $A>0$, and hence is infinite. So consider $A>0$ and recall the
  construction of $H_{\infty}$ from Poisson point measures in the
  stable cases $1<\beta<2$ (the proof can be done in a similar way in
  the Brownian case). By construction of $H_{\infty}$, we will have
  that $H_{\infty}(t)>A$ for $t$ sufficiently large if and only if,
  conditional on the Poisson point measure $\sum_{i \geq 1}
  (1-v_i)r_i\delta_{t_i}$, the number of $i$ such that $t_i>A+1$
  and \[A_i:=\max_{x\in [0,\tau_+^{(i)}((1-v_i)r_i)]}H^{(i)}_+(x) \geq
  t_i-A\] is almost surely finite. By the Borel-Cantelli lemma, it is
  therefore sufficient to check that the sum $\sum_{t_i >
    A+1}\Prob{A_i \geq t_i-A | t_i,(1-v_i)r_i}$ is almost surely
  finite. By standard rules of Poisson measures,
\begin{align*}
\Prob{A_i \geq t_i-A|t_i,(1-v_i)r_i} 
& = 1-\exp{(-(1-v_i)r_i\mathbb N\left[t_i-A\leq H_{\max}\leq t_i\right])}\\
& = 1-\exp{(-(1-v_i)r_ic_{\beta}\mathbb ((t_i-A)^{1/(1-\beta)}-t_i^{1/(1-\beta)}))}\\
& \leq  (1-v_i)r_ic_{\beta}\mathbb
((t_i-A)^{1/(1-\beta)}-t_i^{1/(1-\beta)}).
\end{align*}
Finally,
\begin{align*}
& \E{\sum_{t_i>A+1}(1-v_i)r_i ((t_i-A)^{1/(1-\beta)}-t_i^{1/(1-\beta)})} \\
&
=\int_0^1 \int_0^{\infty} \int_{A+1}^{\infty} (1-v) 
r((t-A)^{1/(1-\beta)}-t^{1/(1-\beta)})
\tfrac{\beta(\beta-1)}{\Gamma(2-\beta)}
\exp{(-c_{\beta}rt^{1/(1-\beta)})}r^{-\beta}\mathrm d t\mathrm d r
\mathrm dv,
\end{align*}
which is clearly finite.  This gives the desired result. The proof is
identical for the behavior as $x \rightarrow -\infty$.
\end{proof}

From their
construction from Poisson point measures, it is immediate that the
processes $\eta^+$ and $\eta^-$ both have independent (but not
stationary) increments. The following lemma is an obvious corollary of the self-similarity of $H_{\infty}$.

\begin{lemma} \label{lem:egalloieta} Let $1<\beta \leq 2$. Then
for all $x\geq0$ and $m>0$,
\[
\left(m^{1/\alpha}\left(\eta^+(u+x)-\eta^+(u)\right) , u \geq 0 \right)\equidist 
\left(\eta^+\left(\frac{u+x}{m}\right)-\eta^+\left(\frac{u}{m}\right),
  u \geq 0\right).
\]
In particular, for any $a > 0$,
\[
m^{1/\alpha}\left(\eta^+(ma+u)-\eta^+(u) \right) \convdist \eta^+(a)
\]
as $m \to \infty$. The same holds by replacing the process $\eta^+$ by $\eta^-$ and then by $\eta$.
\end{lemma}

%%%%%%%%%%%%%%%%%%%%%%%%%%%%%%%%%%%
% Convergence of height processes %
%%%%%%%%%%%%%%%%%%%%%%%%%%%%%%%%%%%

\section{Convergence of height processes}
\label{sec:convheightprocess}

Let $\mathcal{E}^*$ be the set of excursions $f: \R \to \R^+$ such
that $f$ has a unique maximum.  We write $f_{\max} = \max_{x \in \R}
f(x)$ and $x_{\max} = \argmax_{x \in \R} f(x)$.  Let $\phi:
\mathcal{E}^* \to C(\R, \R^+)$ be defined by
\[
\phi(f)(t) = f_{\max} - f(x_{\max} + t).
\]
Let $(H(x), 0 \leq x \leq \sigma)$ be an excursion with a unique
maximum.  Extend this to a function in $\mathcal{E}^*$ by setting $H$
to be zero outside the interval $[0,\sigma]$.  Now put
\[
\bar{H} = \phi(H).
\]
The aim of this section is to prove Theorem \ref{thm:stablefragfonc},
which, using the scaling property of the stable height process, can be
re-stated as

\begin{theorem} \label{thm:stablefragfonc2}
Let $H_x$ have the distribution of $\bar{H}$ under $\N(\cdot | \sigma
= x)$.  Then
\[
H_x \convdist H_{\infty}
\]
as $x \to \infty$, in the sense of uniform convergence on compact intervals.
\end{theorem}

Write $H^{(m)}_{\infty}$ for the function which is
$H_{\infty}$ capped at level $m$:
\[
H^{(m)}_{\infty}(x) 
= \begin{cases}
H_{\infty}(x) & \text{if $-\eta^{-}(m) \leq x \leq \eta^+(m)$} \\
m & \text{otherwise}.
\end{cases}
\]
Then we can re-state Theorem~\ref{thm:CVHcenter} as

\begin{theorem} \label{thm:CVHcenter2}
For all $m > 0$,
\[
H_{\infty}^{(m)} \equidist \bar{H} \text{ under $\N(\cdot |
H_{\max} = m)$}.
\]
\end{theorem}

We will need the following technical lemma.

\begin{lemma} \label{lem:etaHindep}
For all $a > 0$,
\[
(m^{1/\alpha} \eta(ma), H_{\infty}^{(ma)}) \convdist (\tilde{\eta}(a), H_{\infty}),
\]
as $m \to \infty$, where $\tilde{\eta}(a)$ has the same distribution as $\eta(a)$ and is independent of $H_{\infty}$.  Here, the convergence is for the topology associated
with the metric
\[
d((a_1,f_1), (a_2,f_2)) = |a_1 - a_2| + \sum_{k \in \N} 2^{-k} \left(\sup_{t \in
  [-k,k]} |f_1(t) - f_2(t)| \wedge 1 \right)
\]
on $\R \times C(\R, \R^+)$.
\end{lemma}

\begin{proof}
  Let $f: \R \to \R$ be a bounded continuous function and let
  $g:C([-n,n], \R^+) \to \R$ be a bounded continuous function for some
  $n \in \N$.  To ease notation, whenever $h: \R \to \R^+$ we will
  write $g(h)$ for $g(h|_{[-n,n]})$.

It follows from Lemma \ref{lem:egalloieta} that $m^{1/\alpha} \eta(ma)
\equidist \eta(a)$ for all $m,a > 0$, but this is insufficient to give
the required asymptotic independence.  Since the processes
$\eta^+,\eta^-$ have independent increments, by construction we have
that $m^{1/\alpha}(\eta(ma) - \eta(u))$ is independent of
$(H_{\infty}(y), y \in [-\eta^-(u), \eta^+(u)])$ and $\eta^+(u),
\eta^-(u)$, whenever $ma > u$.  By Lemma~\ref{lem:egalloieta}, 
\[
m^{1/\alpha}(\eta(ma) - \eta(u)) \convdist \eta(a)
\]
as $m \to \infty$.  So when $ma > u$ we certainly have
\begin{align*}
& \E{f(m^{1/\alpha}(\eta(ma) - \eta(u)))
     g(H_{\infty}^{(ma)}) \I{\eta^-(u) > n, \eta^+(u) > n}} \\
& = \E{f(m^{1/\alpha}(\eta(ma) - \eta(u)))} 
    \E{g(H_{\infty}^{(ma)}) \I{\eta^-(u) > n, \eta^+(u) > n}} \\
& \to \E{f(\eta(a))}
    \E{g(H_{\infty}) \I{\eta^-(u) > n, \eta^+(u) > n}}.
\end{align*}
Without loss of generality, the functions $|f|$ and $|g|$ are bounded
by 1.  To ease notation, put $G_m = g(H_{\infty}^{(ma)}) \I{\eta^-(u) > n,
  \eta^+(u) > n}$.  We will first show that
\[
\left| \E{f(m^{1/\alpha}(\eta(ma) - \eta(u))) G_m} - \E{f(m^{1/\alpha}
  \eta(ma)) G_m} \right| \to 0
\]
as $m \to \infty$.  Clearly, this absolute value is smaller than 
\begin{equation*}
\E{\left| f(m^{1/\alpha}(\eta(ma) - \eta(u)))-f(m^{1/\alpha}
  \eta(ma)) \right| }  = \E{\left| f((\eta(a) - \eta(u/m)))-f(
  \eta(a)) \right| } ,   
\end{equation*}
by the self-similarity property of $\eta$ (Lemma \ref{lem:egalloieta}).
The function $f$ is bounded and continuous and so, by dominated
convergence, this last quantity converges to $0$.  So we obtain that
\[
\E{f(m^{1/\alpha}\eta(ma)) g(H_{\infty}^{(ma)})
\I{\eta^-(u) > n, \eta^+(u) > n}} 
\to \E{f(\eta(a))} \E{g(H_{\infty})
 \I{\eta^-(u) > n, \eta^+(u) > n}}.
\]
We must now remove the $\I{\eta^-(u) > n, \eta^+(u) > n}$.  Again
using self-similarity, we have that $\eta^-(u) \to \infty$ and $\eta^+(u) \to \infty$
almost surely as $u \to \infty$.  Take $\epsilon > 0$.  Then there
exists a $u$ such that
\[
\Prob{\eta^-(u) \leq n} < \frac{\epsilon}{2}, 
\quad \Prob{\eta^+(u) \leq n} < \frac{\epsilon}{2}.
\]
So for all $m > 0$,
\begin{align*}
& \left| \E{f(m^{1/\alpha}\eta(ma)) g(H_{\infty}^{(ma)}) 
\I{\eta^-(u) > n, \eta^+(u) > n}}  
- \E{ f(m^{1/\alpha}\eta(ma)) g(H_{\infty}^{(ma)})} \right| \\
& \leq \Prob{\eta^-(u) \leq n} + \Prob{\eta^+(u) \leq n} < \epsilon.
\end{align*}
It is now straightforward to conclude that
\begin{equation*}
\E{f(m^{1/\alpha} \eta(ma)) g(H_{\infty}^{(ma)})}
\to
\E{f(\eta(a))} \E{g(H_{\infty})}
\end{equation*}
as $m \to \infty$.
\end{proof}

\begin{proof}[Proof of Theorem~\ref{thm:stablefragfonc2}.]  Let $g:
C([-n,n],\R^+) \to \R^+$ be a bounded continuous function for some $n
\in \N$.  As before, if $h: \R \to \R^+$, we will write $g(h)$ for
$g(h|_{[-n,n]})$.  Then it will be sufficient for us to show that
\[
\E{g(H_x)} \to \E{g(H_{\infty})}
\]
as $x \to \infty$.

Without loss of generality, take $|g| \leq 1$.  Then
\begin{align*}
\E{g(H_x)} & = \N[g(\bar{H}) | \sigma = x] \\
& = \N[g(\phi(H)) | \sigma = x] \\
& = c \int_{\R^+} \N[g(\phi(H^{[x/\sigma]})) \I{\sigma \geq x} | H_{\max}
= x^{-\alpha} u] u^{1/\alpha} \mathrm{d}u,
\end{align*}
by Lemma~\ref{lem:sigmaH_max}. Suppose that we could show that
\begin{equation}
\N[g(\phi(H^{[x/\sigma]}))\I{\sigma \geq x} | H_{\max} = x^{-\alpha} u]
 \to \E{g(H_{\infty})} \mathbb N[\I{\sigma \geq 1} | H_{\max} = u]
 \label{eqn:conv}
\end{equation}
as $x \to \infty$ for all $u > 0$. Since $|g| \leq 1$,
\[
\N[g(\phi(H^{[x/\sigma]}))\I{\sigma \geq x} | H_{\max} = x^{-\alpha} u]
\leq \N[\I{\sigma \geq x} | H_{\max} = x^{-\alpha} u]
= \N[\I{\sigma \geq 1} | H_{\max} = u],
\]
by Lemma \ref{lem:scaling} and we also have that
\[
\int_{\R^+} \N[\I{\sigma \geq 1} | H_{\max} = u]
u^{1/\alpha} \mathrm{d} u = c \N[\sigma \geq 1] <
\infty. 
\]
So then by the dominated convergence theorem, we would be able to
conclude that
\[
\N[g(\phi(H)) | \sigma = x] \to \E{g(H_{\infty})}.
\]
It remains to prove (\ref{eqn:conv}).  By Theorem~\ref{thm:CVHcenter2}, 
\begin{align*}
& \N[g(\phi(H^{[x/\sigma]})) \I{\sigma \geq x} | H_{\max} = x^{-\alpha}
u] \\
& = \E{g((x^{-1} \eta(x^{-\alpha} u))^{\alpha} H_{\infty}^{(x^{-\alpha}
    u)}(x^{-1} \eta(x^{-\alpha} u)\ \cdot\ )) \I{\eta(x^{-\alpha} u)
    \geq x}}.
\end{align*}
Now, by Lemma~\ref{lem:etaHindep} we have that
\[
(x^{-1} \eta(x^{-\alpha} u), H_{\infty}^{(x^{-\alpha}
    u)}) \convdist (\tilde{\eta}(u), H_{\infty})
\]
as $x \to \infty$, where $\tilde{\eta}(u)$ and $H_{\infty}$ are
independent.  By the Skorokhod representation theorem, we may suppose that this convergence is almost sure.  But then, by the bounded convergence theorem, since $\Prob{\tilde{\eta}(u) = 1} = \N(\sigma = 1 | H_{\text{max}} = u) = 0$, we have
\begin{align*}
& \E{g((x^{-1} \eta(x^{-\alpha} u))^{\alpha} H_{\infty}^{(x^{-\alpha}
    u)}(x^{-1} \eta(x^{-\alpha} u)\ \cdot\ )) \I{\eta(x^{-\alpha} u)
    \geq x}} \\
& \to \E{g(\tilde{\eta}(u)^{\alpha} H_{\infty}(\tilde{\eta}(u)\ \cdot\
  )) \I{\tilde{\eta}(u) \geq 1}},
\end{align*}
as $x \to \infty$.  By the scaling property of $H_{\infty}$ (Lemma
\ref{lem:propofH_infty}) and the independence of $\tilde{\eta}(u)$ and
$H_{\infty}$, we see that this last is equal to
\[
\E{g(H_{\infty})} \Prob{\tilde{\eta}(u) \geq 1}.
\]
Since $\Prob{\tilde{\eta}(u) \geq 1} = \N[\I{\sigma \geq 1} | H_{\max}
  = u]$, the result follows.
\end{proof}

%%%%%%%%%%%%%%%%%%%%%%%%%%%%%%%%%%%%%%%%%%%%%%%%%%%%%%%%%%%%%%%%%%
% Convergence of open sets and their sequences of ranked lengths %
%%%%%%%%%%%%%%%%%%%%%%%%%%%%%%%%%%%%%%%%%%%%%%%%%%%%%%%%%%%%%%%%%%

\section{Convergence of open sets and their sequences of ranked lengths} 
\label{sec:proof}

In this section, we prove Theorem \ref{thm:stablefrag} and Corollary
\ref{corostablerank}.

We will need the concept of a \emph{tagged fragment}, that is, the
size $(\lambda(t))_{t \geq 0}$ of a block of the fragmentation which
is tagged uniformly at random.  Then, according to \cite{BertoinSSF},
$(-\log(\lambda(t)), t \geq 0)$ is a time-change of a subordinator
$\xi$ with Laplace exponent
\begin{equation} \label{eqn:Laplacetagged}
\phi(t) = \int_{\Sfl} \left(1 - \sum_{i=1}^{\infty} s_i^{1+t}\right) \nu(\mathrm
d \mathbf s), \quad t \geq 0.
\end{equation}
More specifically, $\lambda(t) = \exp(-\xi(\rho(t)))$, where
\[
\rho(t) := \inf\left\{u \geq 0 : \int_0^u \exp(\alpha \xi(r)) \mathrm{d}r \geq
  t\right\}.
\]

\begin{lemma} 
\label{lemmaconditions}
(i) For all $a>0$,
\[
\mathrm{Leb}\{x \in \mathbb R : H_{\infty}(x)=a\}=0 \text{ a.s. }
\] 
(ii) For all $a>0$, almost surely, a is not a local maximum of $H_{\infty}$.
\end{lemma}

\begin{proof} (i) By Proposition 1.3.3 of Duquesne and Le
  Gall~\cite{Duquesne/LeGall}, the height process $(H(x), x \geq 0)$
  possesses a collection of local times $(L_s^x, s \geq 0, x \geq 0)$,
  which we can assume is jointly measurable, and continuous and
  non-decreasing in $s$.  Moreover, for any non-negative measurable
  function $g: \R^+ \to \R^+$,
\[
\int_0^s g(H_r) \mathrm{d} r  = \int_{\R^+} g(x) L_s^x \mathrm{d} x
\quad \text{a.s.}
\]
Taking $g(x) = \I{x=a}$, we see that for any $s \geq 0$,
\[
\mathrm{Leb}\{t \in [0,s]: H(t) = a\} = 0 \quad \text{a.s.}
\]
Since the height process is built out of excursions, the same is true
for $H$ under the excursion measure $\N$.  In other words, for all $a
> 0$,
\begin{equation}
\label{leb0}
\N(\{\mathrm{Leb}\{x \in [0,\sigma]: H(x) = a\}\neq 0\}) = 0.
\end{equation}
Moreover, from the construction of the process $H_{\infty}$ in
Section~\ref{sec:Williams}, and using the same notation, we see
that
\[
\mathrm{Leb}\{x \in \R: H_{\infty}(x) = a \}
= \sum_{i \geq 1} \sum_{j: u_j^{(i)} \leq r_i} \mathrm{Leb}\left\{x
\in [0, \sigma(f_j^{(i)})]: f_j^{(i)}(x) = t_i - a\right\},
\]
where $\sum_{i \geq 1} \delta_{(r_i,t_i)}$ is a Poisson point measure
of intensity 
\[
\beta (\beta - 1) (\Gamma(2-\beta))^{-1}
\exp(-rc_{\beta} t^{1/(1-\beta)}) r^{-\beta} \mathrm{d}r \mathrm{d}t
\]
on $\R^+\setminus\{0\} \times \R^+\setminus \{0\}$ and, for each $i
\geq 1$, $\{(u^{(i)}_j, f^{(i)}_j), j \geq 1\}$ are the points of a
Poisson point measure of intensity $\mathrm{d}u \N(\cdot,H_{\max} \leq
t_i)$ on $\R^+ \times \mathcal{E}$. From this representation and
(\ref{leb0}), it is clear that the Lebesgue measure of $\{x \in \R:
H_{\infty}(x) = a \}$ is equal to 0 almost surely.

(ii) By Lemma 2.5.3 of Duquesne and Le Gall~\cite{Duquesne/LeGall}, for all $a > 0$,
\[
\N(\{\text{$a$ is a local minimum of $H$}\}) = 0.
\]
With notation as in part (i), we have
\[
\{ \text{$a$ is a local maximum of $H_{\infty}$} \} = \bigcup_{i \geq
  1} \left(\{t_i=a \} \bigcup_{j: u_j^{(i)} \leq r_i} \left\{\text{$t_i - a$ is a local
    minimum of $f_j^{(i)}$}\right\} \right)
\]
(recall that in the construction of $H_{\infty}$, the excursions
$f_j^{(i)}$ are glued upside down at levels $t_i$, so that local
minima of these excursions correspond to local maxima of
$H_{\infty}$). The conclusion is now obvious.
\end{proof}

The proof of Theorem \ref{thm:stablefrag} follows immediately from
Theorem \ref{thm:stablefragfonc2}, Proposition
\ref{prop:fnconvimpliessetconv} and Lemma \ref{lemmaconditions}.

It remains to prove Corollary \ref{corostablerank}. We will need the
following generalization of Theorem \ref{thm:stablefragfonc2} and the
forthcoming Lemma \ref{lemmarestriction}.

\begin{theorem} \label{thm:heightprocessconvdbis} 
  Let $H_x$ be as in Theorem \ref{thm:stablefragfonc2}, and denote by
  $H_x|_K$ its restriction to the compact $[-K,K]$ (and similarly for
  $H_{\infty})$. Then, for all $K>0$,
\[
(H_x|_K, \mathrm{Leb}\{ y \in [-K,K] : H_x(y)<1\}) \convdist
(H_{\infty}|_K, \mathrm{Leb}\{ y \in [-K,K] : H_{\infty}(y)<1\}), 
\]
as $x \to \infty$.  Here, the convergence is in the topology associated with the metric
\[
d((f_1, a_1), (f_2, a_2)) = \sup_{x \in [-K,K]} |f_1(x) - f_2(x)| + |a_1 - a_2|
\]
on $C([-K,K], \R^+) \times \R^+$.
\end{theorem}

\begin{proof}  By the Skorokhod representation theorem, there exists a probability space such that the convergence in Theorem~\ref{thm:stablefragfonc2} is almost sure.  Then the result follows from Lemma \ref{lem:cvLeb} and Lemma \ref{lemmaconditions} (i).
\end{proof}

\begin{lemma} 
\label{lemmarestriction}
Given $\epsilon>0$, there exists $K>0$ such that 
\[
\Prob{\inf_{y \in (-\infty,K] \cup [K,+\infty)} H_x(y)<1} \leq \epsilon
\]
for all $x$ sufficiently large.
\end{lemma}

\begin{proof}
  By symmetry, it is sufficient to show that there exists $K>0$
  such that \linebreak $\Prob{\inf_{y \in[K,+\infty)} H_x(y)<1} \leq
  \epsilon$ for all $x$ sufficiently large.  As in the proof of
  Theorem \ref{thm:stablefragfonc2}, we combine Lemma \ref{lem:sigmaH_max}
  and Theorem \ref{thm:CVHcenter} to get
\begin{align} \label{imp}
& \Prob{\inf_{y \in[K,+\infty)} H_x(y)<1}  \\
& \qquad = c\int_{\mathbb R^+} \Prob{\inf_{y \in[K,+\infty)}
  H_{\infty}^{(x^{-\alpha}u)}(x^{-1}\eta(x^{-\alpha}u)y)<x^{\alpha}
\eta^{-\alpha}(x^{-\alpha}u),\eta(x^{-\alpha}u) \geq x}
u^{1/\alpha}\mathrm d u. \nonumber
\end{align}
For all $L \geq 1$, the probability in the integral can be bounded
above by
\[
\Prob{\inf_{y \in[x^{-1}\eta(x^{-\alpha}u)K,+\infty)}
  H_{\infty}^{(x^{-\alpha}u)}(y)<x^{\alpha}\eta^{-\alpha}(x^{-\alpha}u),
  Lx > \eta(x^{-\alpha}u) \geq x} 
+ \Prob{\eta(x^{-\alpha}u) \geq Lx}.
\]
The first probability in this sum is in turn smaller than 
\[
\Prob{\inf_{y \in[K,+\infty)}
  H_{\infty}^{(x^{-\alpha}u)}(y)<L^{-\alpha}, Lx > \eta(x^{-\alpha}u)
  \geq x},
\]
and so is also smaller than
\begin{equation} \label{Proba1}
\Prob{\inf_{y \in[K,+\infty)}
  H_{\infty}^{(x^{-\alpha}u)}(y)<L^{-\alpha}}
 \wedge \Prob{Lx > \eta(x^{-\alpha}u) \geq x}.
\end{equation}
Recall that by self-similarity, 
\[
\Prob{Lx > \eta(x^{-\alpha}u) \geq x}
=\Prob{L > \eta(u) \geq 1} \leq \Prob{\eta(u) \geq 1}.
\]
Hence, (\ref{Proba1}) can be bounded from above by $\Prob{\eta(u)
\geq 1}$ when $x^{-\alpha}u \leq L^{-\alpha}$, and by
\[
\Prob{\inf_{y \in[K,+\infty)} H_{\infty}(y)<L^{-\alpha}} 
\wedge \Prob{\eta(u) \geq 1}
\]
when $x^{-\alpha}u > L^{-\alpha}$. From the identity (\ref{imp}) and
these observations, we have
\begin{align} 
\Prob{\inf_{y \in[K,+\infty)} H_x(y)<1} \leq & \
c \int_0^{L^{-\alpha}x^{\alpha}} \Prob{\eta(u) \geq 1}
u^{1/\alpha}\mathrm d u \nonumber \\
& + c\int_0^{\infty} \Prob{\inf_{y \in[K,+\infty)}
  H_{\infty}(y)<L^{-\alpha}} \wedge \Prob{\eta(u) \geq
1} u^{1/\alpha}\mathrm d u \label{imp2} \\ 
& + c\int_0^{\infty}  \Prob{\eta(u) \geq L}
u^{1/\alpha}\mathrm d u. \nonumber
\end{align}
Recall that $c\int_0^{\infty} \Prob{\eta(u) \geq L}
u^{1/\alpha}\mathrm d u=\mathbb N[\sigma \geq
L]<\infty$. Then fix $L$ large enough that the third integral in the
right-hand side of (\ref{imp2}) is smaller than $\epsilon/3$. This $L$
being fixed, note that the first integral on the right-hand side of
(\ref{imp2}) is smaller than $\epsilon/3$ for all $x$ sufficiently
large.  Since $H_{\infty}(y) \to \infty$ a.s.\ as $y \to +\infty,
-\infty$, by dominated convergence we have that the second integral
(which does not depend on $x$) is smaller than $\epsilon/3$ for $K$
sufficiently large.  Hence, there exists some $K>0$ such that
$\Prob{\inf_{y \in[K,+\infty)} H_x(y)<1} \leq \epsilon$ for all $x$
sufficiently large.
\end{proof}

\begin{proof}[Proof of Corollary \ref{corostablerank}.]  Let $\epsilon>0$
and consider some real number $K$ such that, for $x$ large enough, the
events $E_x=\{\inf_{y \in (-\infty,K] \cup [K,+\infty)} H_x(y) < 1
\}$ all have probability smaller than $\epsilon/5$ (such a $K$ exists
by Lemma \ref{lemmarestriction}). Taking $K$ larger if necessary, we
also have that $E_{\infty}$ (defined in a similar way using
$H_{\infty}$) has probability smaller than $\epsilon/5$.

Re-stated in terms of the functions $H_x$, our goal is to check that
the decreasing sequence of lengths of the interval components of
$O_x:=\{y\in \mathbb R: H_x(y)<1\}$, say $|O_x|^{\downarrow}$,
converges in distribution as $x \rightarrow \infty$ to the decreasing
sequence of lengths, $|O_{\infty}|^{\downarrow}$, of the interval
components of $O_{\infty}:=\{y\in \mathbb R : H_{\infty}(y)<1\}$. We
recall that the topology considered on $\Sflfin$ is given by the
distance $d_{\Sflfin}(\mathbf s,\mathbf
{s'})=\sum_{i=1}^{\infty}|s_i-s_i'|.$ Now, let $O_x^K$ be the
restriction of $O_x$ to the open set $(-K,K)$, for $x \in
(0,\infty]$.  Let $|O^K_x|^{\downarrow}$ be the corresponding ranked
sequence of lengths of interval components. By the Skorokhod
representation theorem, we may suppose that the convergence stated in
Theorem \ref{thm:heightprocessconvdbis} holds almost surely. From
this, Lemma \ref{lemmaconditions} and Proposition
\ref{prop:fnconvimpliessetconv}, we have that $O_x^K$ converges to
$O_{\infty}^K$, in the sense that their complements in $[-K,K]$ converge in the Hausdorff distance.
Moreover,
the total length of $O_x^K$ converges to that of $O_{\infty}^K$. We
deduce that $|O^K_x|^{\downarrow}$ converges to
$|O_{\infty}^K|^{\downarrow}$ in the \emph{pointwise} distance on
$\Sflfin$ (see Proposition 2.2 of Bertoin~\cite{BertoinBook}). But
since we also have convergence of the total length of the open sets,
the convergence actually holds in the $d_{\Sflfin}$ distance.

Now, let $f: \Sflfin \rightarrow \mathbb R$ be any continuous function
such that $\sup_{\mathbf s \in \Sflfin}|f(\mathbf s)| \leq 1$. Since
$O_x=O^K_x$ on $E_x^{\mathrm{c}}$, we have
\begin{align*}
\Bigg|\E{f(|O_x|^{\downarrow})}-\E{f(|O_{\infty}|^{\downarrow})} \Bigg|
& \leq \Bigg|\E{f(|O^K_x|^{\downarrow})
  \mathbbm{1}_{E_x^{\mathrm{c}}}}-\E{f(|O^K_{\infty}|^{\downarrow})
\mathbbm{1}_{E_{\infty}^{\mathrm{c}}}} \Bigg|
+ \Prob{E_x} + \Prob{E_{\infty}} \\
& \leq \Bigg|\E{f(|O^K_x|^{\downarrow})}-
\E{f(|O^K_{\infty}|^{\downarrow})} \Bigg| + 2\Prob{E_x} + 2
\Prob{E_{\infty}}.
\end{align*}
Using the fact that $|O^K_x|^{\downarrow}$ converges in distribution
to $|O^K_{\infty}|^{\downarrow}$, we get that for all $x$ sufficiently
large, $\left|\E{f(|O_x|^{\downarrow})}-
  \E{f(|O_{\infty}|^{\downarrow})}\right| \leq \epsilon$. The result
follows.
\end{proof}

%%%%%%%%%%%%%%%%%%%%%%%%%%%%%%%%%
% Behavior of the last fragment %
%%%%%%%%%%%%%%%%%%%%%%%%%%%%%%%%%

\section{Behavior of the last fragment}
\label{seclastfrag}

In this section, we prove the results stated in Corollary
\ref{lastfrag} and the remark which follows it.

First, it is implicit in the proofs in the previous sections that the
distribution of the length of the interval component of $\{y\in
\mathbb R :H_{x}(y)<1\}$ that contains $0$ (i.e.\ the distribution of
$xF_{*}((\zeta-x^{\alpha})^+)$) converges as $x \rightarrow \infty$ to
the distribution of the length of the interval component of $\{y \in
\R :H_{\infty}(y)<1\}$ that contains $0$.  By construction of the
function $H_{\infty}$ (see Section~\ref{sec:Williams}), this interval
component is distributed as $\eta(1)$.  Indeed, almost surely $1$ is
not one of the $t_i$'s and therefore $H_{\infty}(y)<1$ for every $y
\in (-\eta^-(1), \eta^+(1))$.  Moreover, it is easy to see that
$H_{\infty}(\eta^+(1))=H_{\infty}(\eta^-(1))=1$. In other words, we
have that
\[
t \left(F_{\ast}\left( (\zeta -t)^{+}\right)\right)^{\alpha}
\convlaw (\eta(1))^{\alpha} \text{ as $t \rightarrow 0$}.
\]
Recall then that $(\eta(1))^{\alpha}$ is distributed as
$\sigma^{\alpha}$ under $\mathbb N(\cdot | H_{\max}=1)$ which, by
Lemma~\ref{lem:sigmaH_max}, is easily seen to be distributed as the
$(-1/\alpha-1)$-size-biased version of $\zeta$ defined in
(\ref{eqn:biasedzeta}).

We now turn to the bounds given in (i) and (ii), Corollary
\ref{lastfrag}, for the tails of this size-biased version of $\zeta$.

(i) When $t \rightarrow \infty$, the given bounds are easy
consequences of Proposition 14 of \cite{HaasLossMass} on the
asymptotic behavior of $\Prob{\zeta>t}$ as $t \rightarrow
\infty$. Indeed, according to that result, there exist $A,B>0$
such that, for all $t$ large enough,
\[
\exp(-B\Psi(t)) \leq \Prob{\zeta \geq t} \leq \exp(-A\Psi(t)),
\]
where $\Psi$ is the inverse of the bijection $t\in[1,\infty)
\rightarrow t/\phi(t) \in [1/\phi(1),\infty)$ and $\phi$ is defined at
(\ref{eqn:Laplacetagged}).  Miermont \cite[Section 3.2]{Miermont}
shows that in the case of the stable fragmentation,
\[
\phi(t)=(1+\alpha)^{-1}\frac{\Gamma(t-\alpha)}{\Gamma(t)}.
\] 
Now let $c$ be a positive constant that may vary from line to
line. Using the fact that $\Gamma(z)$ is proportional to
$z^{z-1/2}\exp(-z)$ for large $z$, we get that $\phi(t)\sim
ct^{-\alpha}$ as $t \rightarrow \infty$.  So $\Psi(t)\sim
ct^{1/(1+\alpha)}$ as $t \rightarrow \infty$.  We just need to convert the
statements about the tails of $\zeta$ into statements about the tails
of $\zeta_{*^{\alpha}}$. On the one hand, note that
\[
\Prob{\zeta_{*^{\alpha}} \geq t} 
=  \frac{\E{\zeta^{-1/\alpha-1}\mathbbm{1}_{\{\zeta \geq t\}}}}{\E{\zeta^{-1/\alpha-1}}}
\geq \frac{t^{-1/\alpha-1} \Prob{\zeta \geq t} }{\E{\zeta^{-1/\alpha-1}}}
\geq \frac{\exp(-ct^{1/(1+\alpha)})}{\E{\zeta^{-1/\alpha-1}}} , 
\]
for $t$ sufficiently large. On the other hand, using the Cauchy-Schwarz
inequality and the fact that $\zeta$ has positive moments of all
orders, we easily get that
\[
\Prob{\zeta_{*^{\alpha}} \geq t}
=  \frac{\E{\zeta^{-1/\alpha-1}\mathbbm{1}_{\{\zeta \geq t\}}}}{\E{\zeta^{-1/\alpha-1}}}
\leq \frac{c \ \Prob{\zeta \geq t}^{1/2}}{\E{\zeta^{-1/\alpha-1}}} ,
\]
for all $t \geq 0$.  The result follows immediately.

(ii)  The second assertion is as straightforward to prove. First,
\begin{equation} \label{eqn:zetastar}
\Prob{\zeta_{*^{\alpha}} \leq t}
=  \frac{\E{\zeta^{-1/\alpha-1}\mathbbm{1}_{\{\zeta \leq t\}}}}{\E{\zeta^{-1/\alpha-1}}}
\leq \frac{t^{-1/\alpha-1} \Prob{\zeta \leq t}}{\E{\zeta^{-1/\alpha-1}}} . 
\end{equation}
In Section 4.2.1 of \cite{HaasLossMass} it is proved that for all
$q<1+\underline{p}/(-\alpha)$, $\Prob{\zeta \leq t} \leq t^{q}$ for
small $t$, where
\[
\underline{p}=\sup\left\{q \geq 0: \int_{(1,\infty)} \exp(qx)
L(\mathrm d x)<\infty \right\}.
\] 
Here, $L$ is the L\'evy measure of the subordinator $\xi$ with Laplace
exponent $\phi$ (see the beginning of Section~\ref{sec:proof} for the
definition). Using Miermont's results \cite[Section 3.2]{Miermont}
again, the measure $L$ is proportional to
$\exp(x)(\exp(x)-1)^{\alpha-1}\mathbbm{1}_{\{ x>0\}}\mathrm d x$.  It
follows that $\underline{p}=-\alpha$.  Combining this with
(\ref{eqn:zetastar}) gives the desired result.

\bigskip

We finish this section with the proof of equation (\ref{eqPhi}), which
is based on the fact that $\Phi(\lambda):= \E{\exp{(-\lambda
    \eta(1))}}=\E{\exp{(-
    \eta(\lambda^{-\alpha}))}}$, $\lambda \geq 0$.  We recall that it is assumed that $\alpha \in (-1/2,0)$, i.e. $\beta \in (1,2)$. Using the Poisson construction of
$\eta(1)$, the expression for the Laplace transform of a subordinator,
the self-similarity of the process $\eta$ and Campbell's formula for
Poisson point processes, we obtain
\begin{align*}
\Phi(\lambda) 
& = \E{\E{\exp \left(-\sum_{i: t_i
\leq \lambda^{-\alpha}} \left(\tau_+^{(i)}((1-v_i)r_i)+\tau_-^{(i)}(v_ir_i)\right)
 \right) \Bigg| v_i,r_i,t_i,i \geq 1}} \\ 
&= \E{\exp \left(-\sum_{i: t_i \leq \lambda^{-\alpha}} r_i
    \int_0^{\infty}(1-e^{- u}) \mathbb N[\sigma \in \mathrm d
    u, H_{\mathrm{max}}\leq t_i]\right)}  \\
&= \E{\exp \left(-(\beta-1)^{\beta/(\beta-1)} \sum_{i: t_i \leq \lambda^{-\alpha}} r_i
    \int_0^{t_i} \int_0^{\infty}(1-e^{- u}) \mathbb N[\sigma
    \in \mathrm d u | H_{\mathrm{max}}=v] v^{\beta/(1-\beta)} \mathrm
    d v \right) } \\ 
&= \E{\exp \left(-(\beta-1)^{\beta/(\beta-1)} \sum_{i: t_i \leq \lambda^{-\alpha}} r_i
    \int_0^{t_i} \E{1-e^{- \eta(v)}}
    v^{\beta/(1-\beta)}  \mathrm d v \right)} \\ 
&= \E{\exp \left(-(\beta-1)^{\beta/(\beta-1)} \sum_{i: t_i \leq \lambda^{-\alpha}} r_i
    \int_0^{t_i} (1-\Phi ( v^{-1/\alpha}))
    v^{\beta/(1-\beta)}  \mathrm d v \right)} \\
&=  \exp \left(\!\! - \int_{\mathbb R^+ \times [0,\lambda^{-\alpha}]} \! \! 
     (1-e^{-(\beta-1)^{\beta/(\beta-1)}r
    \int_0^t (1-\Phi( v^{-1/ \alpha})) v^{\beta/(1-\beta)}
    \mathrm d v})  \tfrac{\beta(\beta-1)}{\Gamma(2-\beta)}
  e^{-c_{\beta}rt^{1/(1-\beta)}} r^{-\beta} \mathrm d r \mathrm dt \right),
\end{align*}
where $c_{\beta}=(\beta-1)^{1/(1-\beta)}$.  Substituting $\beta = (1+
\alpha)^{-1}$, we obtain the desired expression.

%%%%%%%%%%%%%%%%%%%%%%%%%%%%%%%%%%%%%%%
% Almost sure logarithmic asymptotics %
%%%%%%%%%%%%%%%%%%%%%%%%%%%%%%%%%%%%%%%

\section{Almost sure logarithmic asymptotics}
\label{SectionLog} 

The following result describes the almost sure logarithmic behavior
near the extinction time of the largest fragment and the last fragment
processes. It is valid for general self-similar fragmentations with
parameters $\alpha<0$, $c=0$ and $\nu(\sum_{i=1}^{\infty} s_i<1)=0.$ We
recall that for such fragmentations, the extinction time $\zeta$
possesses exponential moments (see \cite[Prop. 14]{HaasLossMass}).

\begin{theorem}\label{thm:largestfrag}
  (i) Suppose there exists $\eta>0$ such that
  $\int_{\Sfl}s_1^{-\eta}\mathbf1_{\{s_1< 1/2 \}}\nu(\mathrm d \mathbf
  s)<\infty$. Then,
\begin{equation*}
\frac{\log (F_{1}((\zeta -t)^{+}))}{\log(t)}
\overset{\mathrm{a.s.}}{\rightarrow } - 1 / \alpha \text{ as $t
\rightarrow 0$.}
\end{equation*}

\noindent (ii) If, moreover, $\alpha \geq -\gamma_{\nu }:=-\inf
\{\gamma>0 :\lim_{\varepsilon \rightarrow 0} \varepsilon^{\gamma} \nu
(s_{1}<1-\varepsilon ) = 0\}$, then the last fragment process $F_*$ is
well-defined (i.e. the fragmentation $F$ can be encoded into an interval fragmentation for which there exists a unique point $x \in (0,1)$ which is
reduced to dust at time $\zeta$) and
\begin{equation*}
\frac{\log (F_{\ast}((\zeta -t)^{+}))}{\log
(t)}\overset{\mathrm{a.s.}}{\rightarrow }-1/\alpha \text{ as $t
\to 0$.}
\end{equation*}
\end{theorem}

\bigskip

\begin{corollary} 
\label{corologstable}
For the stable fragmentation with index $-1/2 \leq \alpha < 0$ (or,
equivalently, $1<\beta \leq 2$), with probability $1$,
\begin{equation*}
\lim_{t \rightarrow 0}\frac{\log (F_{1}((\zeta -t)^{+}))}{\log(t)}
=\lim_{t \rightarrow 0} \frac{\log (F_{\ast}((\zeta -t)^{+}))}{\log
(t)}= -\frac{1}{\alpha} = \frac{\beta}{\beta-1}.
\end{equation*}
\end{corollary}

\bigskip

\begin{proof}[Proof of Theorem \ref{thm:largestfrag}] (i) We first
  prove this result in the case where there exists some real number
  $a>0$ such that $\nu(s_1<a)=0.$ We will then explain how to adapt it
  to the more general assumption $\int_{\Sfl}s_1^{-\eta}\mathbf
  1_{\{s_1< 1/2\}}\nu(\mathrm d \mathbf s)<\infty$.  By the first
  Borel-Cantelli lemma, it is sufficient to show that, for any
  $\epsilon > 0$,
\begin{align}
\sum_{i=1}^{\infty} \Prob{ F_1((\zeta - e^{-i})^{+}) >
\exp((\alpha^{-1} + \epsilon) i )}
& < \infty  \label{eqn:BC1} \\
\sum_{i=1}^{\infty} \Prob{ F_1((\zeta - e^{-i})^{+}) \leq
\exp((\alpha^{-1} - \epsilon) i) } & < \infty.  \label{eqn:BC2}
\end{align}
(Note that if $i^{-1}\log (F_{1}((\zeta -e^{-i})^{+}))$ converges
almost surely to $1/\alpha$ as $i \to \infty$, then almost surely for all
sequences $(t_n,n\geq 0)$ converging to 0, $\log (F_{1}((\zeta
-t_n)^{+})) / \log(t_n)$ converges to $-1/\alpha$, since
\[
\frac{F_1((\zeta - e^{-(i_n+1)})^{+})}{-i_n} \leq \frac{F_1((\zeta -
  t_n)^{+})}{\log (t_n)} \leq \frac{F_1((\zeta - e^{-i_n})^{+})}{-(i_{n}+1)},
\]
where $i_n = \fl{-\log(t_n)}$).

Let $\mathcal{F}_t = \sigma(F(s): s \leq t)$ and suppose that $T$ is a
$(\mathcal{F}_t)_{t \geq 0}$-stopping time such that $T < \zeta$
a.s. According to \cite{BertoinSSF}, the branching and self-similarity
properties of $F$ hold for $(\mathcal{F}_t)_{t \geq 0}$-stopping
times, hence
\[
\zeta - T = \sup_{j \geq 1} \left\{
F_j(T)^{-\alpha} \zeta_j \right\},
\]
where $(\zeta_j, j \geq 1)$ are independent copies of $\zeta$,
independent of $F(T)$.  Let
\begin{align*}
T_1 & = \inf\{t \geq 0: F_1(t) \leq \exp((\alpha^{-1} + \epsilon)i) \\
T_2 & = \inf\{t \geq 0: F_1(t) \leq \exp((\alpha^{-1} -
\epsilon)i).
\end{align*}
We start by proving (\ref{eqn:BC1}).  We have
\begin{align*}
\Prob{F_1((\zeta - e^{-i})^+) > \exp((\alpha^{-1} + \epsilon)i)}
& = \Prob{T_1 > \zeta - e^{-i}} \\
& = \Prob{\sup_{j \geq 1} F_j(T_1)^{-\alpha} \zeta_j < e^{-i}} \\
& \leq \Prob{F_1(T_1)^{-\alpha} \zeta_1 < e^{-i}}.\end{align*}
Since we have assumed that there exists $a>0$ such that
$\nu(s_1<a)=0$, we have $F_1(T_1) \geq aF_1(T_1-) \geq a
\exp((\alpha^{-1} + \epsilon)i)$ a.s., and so
\begin{align}
\label{key}
\Prob{F_1(T_1)^{-\alpha} \zeta_1 < e^{-i}} & \leq
\Prob{a^{-\alpha}
\exp(-\alpha(\alpha^{-1} + \epsilon)i) \zeta <e^{-i}} \\
\nonumber & = \Prob{\zeta < a^{\alpha} e^{\alpha \epsilon i}}.
\end{align}
Let $(\lambda(t))_{t \geq 0}$ be the size of a tagged fragment as
defined at the beginning of Section \ref{sec:proof}, and let $\xi$ be
the related subordinator. Then consider the first time at which
$\lambda$ reaches $0$, i.e.
\[
\sigma = \inf\{t \geq 0 : \lambda(t) = 0\}=\int_0^{\infty}
\exp(\alpha\xi(r))\mathrm d r.
\]
Of course, $\sigma \leq \zeta$.  Moreover, by Proposition 3.1 of
Carmona, Petit and Yor~\cite{Carmona/Petit/Yor}, $\sigma$ has moments
of all orders strictly greater than $-1$; this implies, in particular, that $\mathbb
E[\zeta^{-\gamma}]<\infty$ for $0<\gamma<1$. So, by Markov's inequality, we have that
\[
\Prob{F_1((\zeta - e^{-i})^+) > \exp((\alpha^{-1} + \epsilon)i)} \leq
a^{\alpha / 2} e^{\alpha \epsilon i / 2} \E{\zeta^{-1/2}},
\]
which is summable in $i$.

Now turn to (\ref{eqn:BC2}).  Arguing as before, we have
\begin{align*}
\Prob{F_1((\zeta - e^{-i})^+) \leq \exp((\alpha^{-1} - \epsilon)i)}
& = \Prob{T_2 \leq \zeta - e^{-i}} \\
& = \Prob{\sup_{j \geq 1} F_j(T_2)^{-\alpha} \zeta_j \geq e^{-i}}.
\end{align*}
Take $q > -1/\alpha$.  Then, since $\zeta_1, \zeta_2, \ldots$ are
independent and identically distributed and independent of
$F(T_2)$,
\begin{align*}
\Prob{\sup_{j \geq 1} F_j(T_2)^{-\alpha} \zeta_j \geq e^{-i}}
& \leq \sum_{j \geq 1} \Prob{F_j(T_2)^{-\alpha} \zeta_j \geq e^{-i}/2} \\
& \leq 2^q\E{\zeta^q} e^{iq} \E{ \sum_{j \geq 1} F_j(T_2)^{-\alpha
q}}.
\end{align*}
The expectation $\E{\zeta^q}$ is finite and, since $-\alpha q - 1
> 0$, we have
\[
\sum_{j \geq 1} F_j(T_2)^{-\alpha q} \leq F_1(T_2)^{-\alpha q - 1}
\sum_{j \geq 1} F_j(T_2) \leq F_1(T_2)^{-\alpha q - 1}.
\]
But then
\[
\Prob{F_1((\zeta - e^{-i})^+) \leq \exp((\alpha^{-1} - \epsilon)i)} \leq
2^q\E{\zeta^q} \exp(i(\epsilon \alpha q - \alpha^{-1} + \epsilon)),
\]
which is summable in $i$ for large enough $q$.

\bigskip

It remains to adapt this proof to the more general case where
$\int_{\Sfl}s_1^{-\eta}\mathbf 1_{\{s_1<
  1/2\}}\nu(\mathrm d \mathbf s)<\infty$ for some $\eta>0$. The key
inequality (\ref{key}) is then no longer valid and we have to check
that $\Prob{F_1(T_1)^{-\alpha} \zeta_1 < e^{-i}}$ is still summable in
$i$. The rest of the proof remains unchanged. So, denote by
$\Delta(T_1)$ the size of the ``multiplicative" jump of $F_1$ at time
$T_1$ (i.e.\ $\Delta(T_1) := F_1(T_1)/ F_1(T_1-)$) and recall that
$\zeta_1$ denotes a random variable with the same distribution as
$\zeta$ and independent of $F_1(T_1)$. Note that we may, and will,
suppose that $\zeta_1$ is independent of the whole fragmentation $F$.
Then,
\begin{align}
\nonumber
&\mathbb P\left(F_1^{-\alpha}(T_1)\zeta_1<e^{-i}\right)\\
\nonumber
& = \mathbb
P\left(F_1^{-\alpha}(T_1-)(\Delta(T_1))^{-\alpha}\zeta_1<e^{-i}\right)
\\ 
\nonumber 
&\leq \mathbb P\left(e^{-\alpha(\alpha^{-1}+\epsilon)i}
  (\Delta(T_1))^{-\alpha}\zeta_1<e^{-i}\right) \\
\label{comp}
&=\mathbb P\left((\Delta(T_1))^{-\alpha}\zeta_1<e^{\alpha \epsilon
    i},\Delta(T_1)<1/2\right)
+\mathbb P\left(( \Delta(T_1))^{-\alpha}\zeta_1<e^{\alpha \epsilon
    i},\Delta(T_1)\geq1/2\right).
\end{align}
The second term in this last line is bounded from above, for all $\gamma>0$, by
\[
e^{\gamma \alpha \epsilon i} \mathbb E[\zeta^{-\gamma}]2^{-\alpha\gamma},
\]
which is finite and summable in $i$ if we take $0<\gamma<1$.  To bound
the first term in (\ref{comp}), introduce $\mathcal
D_1$, the set of jump times of $F_1$. For $t \in \mathcal D_1$, let
$\mathbf{s}(t)$ be the relative mass frequencies obtained by the
dislocation of $F_1(t-)$, and let $\Delta(t):=F_1(t)/F_1(t-)$.  Since the largest fragment coming from $F_1$ at the time of a split may not be the largest block overall after the split,
$s_1(t)\leq \Delta(t)$. Then,
\begin{align*}
&\mathbb P((\Delta(T_1)^{-\alpha}\zeta_1<e^{\alpha \epsilon i},\Delta(T_1)<1/2) \\
&=\mathbb E\left[\sum_{t \in \mathcal D_1}
 \I{(\Delta(t))^{-\alpha}\zeta_1<e^{\alpha \epsilon
      i},\Delta(t)<1/2} \I{F_1(t-) \geq
    e^{(\alpha^{-1}+\epsilon)i}, F_1(t)\leq
    e^{(\alpha^{-1}+\epsilon)i} }\right] \\ 
&\leq e^{\gamma \alpha \epsilon i}\mathbb E\left[\sum_{t \in \mathcal
    D_1}(s_1(t))^{\alpha\gamma} \I{s_1(t)<1/2\}} \I{F_1(t-) \geq
    e^{(\alpha^{-1}+\epsilon)i} }\right] \mathbb E[\zeta_1^{-\gamma}]. 
\end{align*}
for all $\gamma>0$.  The expectation $\mathbb E[\zeta_1^{-\gamma}] $
is finite when $\gamma<1$, which we assume for the rest of the
proof.  Now, the process $(\Sigma(u),u \geq 0)$ defined
by 
\[
\Sigma(u)=\sum_{t \in \mathcal D_1, t \leq u}(s_1(t))^{\alpha \gamma}
\I{s_1(t)<1/2} \I{F_1(t-) \geq e^{(\alpha^{-1}+\epsilon)i}}
\] 
is increasing, c\`adl\`ag, and adapted to the filtration $(\mathcal
F_t,t \geq 0)$ (and hence optional).  So, according to the Doob-Meyer
decomposition \cite[Section VI]{Rogers/Williams2}, \cite[Theorem 1.8,
Chapter 2]{Jacod/Shiryaev}, it possesses an increasing predictable
compensator $(A(u),u \geq 0)$ such that $\mathbb
E[\Sigma(\infty)]=\mathbb E[A(\infty)]$. Moreover, since the
fragmentation process is a pure jump process where each block of size
$m$ splits at rate $m^{\alpha}\nu(\mathrm d \mathbf s)$ into blocks of
lengths $(ms_1,ms_2,...)$, this compensator is given by
\[
 A(u) = \int_0^{u} F_1(t)^{\alpha}\I{F_1(t) \geq
   e^{(\alpha^{-1}+\epsilon)i} } \mathrm dt    \int_{\Sfl}s_1^{\alpha
   \gamma} \I{s_1<1/2} \nu (\mathrm d \mathbf s), \quad u \geq 0.
\]
The second integral in this product is finite for small enough
$\gamma>0$, by assumption. The first is clearly smaller than
$e^{-\delta(\alpha^{-1}+\epsilon)i} \int_0^{s} F_1(t)^{\alpha+\delta}
\mathbf 1_{\{F_1(t)>0\}}\mathrm dt $, for all $\delta \geq 0$. This
leads us to
\begin{equation}
\label{eqmaj}
\mathbb P\left((\Delta(T_1))^{-\alpha}\zeta_1<e^{\alpha \epsilon
  i},\Delta(T_1)<1/2\right) \leq C_{\gamma}e^{(\gamma \alpha \epsilon
-\delta(\alpha^{-1}+\epsilon)) i}  \mathbb E\left [\int_0^{\zeta}
F_1(r)^{\alpha+\delta} \mathrm dr\right], 
\end{equation}
where $C_{\gamma}$ is a constant depending only on $\gamma>0$ and
which is finite provided $\gamma$ is sufficiently small.  To finish,
we claim that for all $0<\delta<-\alpha$, $\mathbb E[\int_0^{\zeta}
F_1(r)^{\alpha+\delta} \mathrm dr]<\infty$ (note this finiteness is
obvious for $\delta \geq -\alpha$, since $\mathbb
E[\zeta]<\infty$). Indeed, from Bertoin \cite{BertoinSSF}, we know
that the $\alpha$-self-similar fragmentation process $F$ (and its
interval counterpart) can be transformed through (somewhat
complicated) time-changes into a $(-\delta)$-self-similar
fragmentation process with same dislocation measure. We refer to
Bertoin's paper for details. In particular, if $|O_x(t)|$ denotes the
length of the fragment containing $x \in (0,1)$ at time $t$ in the
interval $\alpha$-self-similar fragmentation, and if $\zeta_x$ denotes
the first time at which this length reaches 0, we have
 \[
 \int_0^{\zeta_x} |O_x(r)|^{\alpha+\delta}\mathrm
 dr=\zeta_x^{(\delta)} \leq \zeta^{(\delta)}, 
 \]
 where $\zeta_x^{(\delta)}$ is the time at which the point $x$ is
 reduced to dust in the fragmentation with parameter $-\delta$ and
 $\zeta^{(\delta)}:=\sup_x \zeta_x^{(\delta)}$ is the time at which
 the whole fragmentation with parameter $-\delta$ is reduced to
 dust. Now, since $\alpha+\delta<0$, we have $(F_1(r))^{\alpha+\delta}
 \leq |O_x(r)|^{\alpha+\delta}$ for all $x \in (0,1)$, $r \geq 0$ and,
 therefore,
\[
\mathbb E \left [\int_0^{\zeta} F_1(r)^{\alpha+\delta} \mathrm dr
\right]=\mathbb E \left [\sup_x \int_0^{\zeta_x}
  F_1(r)^{\alpha+\delta} \mathrm dr \right] \leq \mathbb E \left
  [\sup_x \int_0^{\zeta_x} |O_x(r)|^{\alpha+\delta}\mathrm dr \right]
\leq \mathbb E [ \zeta^{(\delta)}]<\infty.
\]
Hence, if we choose $\delta$ small enough that $\gamma \alpha
\epsilon-\delta(\alpha^{-1}+\epsilon)<0$, we get from (\ref{eqmaj})
that $\mathbb P((\Delta(T_1))^{-\alpha}\zeta_1<e^{\alpha \epsilon
  i},\Delta^{(1)}(T_1)<1/2)$ is summable in $i$.

\smallskip

(ii) The additional assumption on $\nu $ implies that it is infinite,
hence we know (see Haas and Miermont~\cite{Haas/Miermont}) that the
fragmentation can be encoded into a continuous function $G$ which is,
moreover, $\gamma$-H\"{o}lder, for all $\gamma <(-\alpha) \wedge
\gamma_{\nu}$ ($= -\alpha$ here). In particular, the maximum, $\zeta$,
of $G$ is attained for some $x \in (0,1)$. More precisely, we claim it
is attained at a \textit{unique} point, which is denoted by
$x_*$. See the end of the proof for an explanation of this
uniqueness. It implies, in particular, that the last fragment process
is well-defined: for each $t < \zeta$, we denote by $O_*(t)$ the
interval component of $\{x \in (0,1) : G(x)>t\}$ which contains
$x_*$, and by $F_*(t)$ the length of this interval. For $t<\zeta $,
let
\begin{align*}
x_{t}^{-} & = \sup\{x \leq x_*: G(x) \leq \zeta - t\} \\
x_{t}^{+} & = \inf\{x \geq x_*: G(x) \leq \zeta - t\},
\end{align*}
so that $O_{\ast }(\zeta -t)= (x_{t}^{-},x_{t}^{+})$.  Then, for all
$0\leq \gamma <-\alpha $, there exists some constant $C$ such that
\begin{align*}
t & = G(x_*)-G(x_{t}^{-})\leq C(x_*-x_{t}^{-})^{\gamma} \\
t & = G(x_*)-G(x_{t}^{+})\leq C(x_{t}^{+}-x_*)^{\gamma},
\end{align*}
and, consequently, $F_{\ast}(\zeta -t)=x_{t}^{+}-x_{t}^{-}\geq
2(t/C)^{1/\gamma}$.  This implies that
\begin{equation*}
\limsup_{t\rightarrow 0}\left( \frac{\log (F_{\ast}((\zeta
-t)^{+}))}{\log (t)} \right) \leq -1/\alpha .
\end{equation*}
For the liminf, use part (i) and the fact that $F_*(t) \leq F_1(t)$
for all $t \geq 0$.

Finally, we have to prove that there is a unique $x \in (0,1)$ such that
$G(x)=\zeta$. Note that
\begin{align*}
&\mathbb{P}\left( \exists t<\zeta :\text{at least two fragments present at
time }t\text{ die at }\zeta \right) \\
&=\mathbb{P}\left( \exists t\in \mathbb{Q}, t < \zeta :%
\text{at least two fragments present at time }t\text{ die at }\zeta \right)
\end{align*}
and this latter probability is equal to $0$ if, for all $t\in
\mathbb{Q}$,
\begin{align*}
&\mathbb{P}\left( \text{at least two fragments present at time }t\text{ die
at }\zeta, t <\zeta\right) \\
&=\mathbb{P}\left( \exists i \neq j:F_{i}^{-\alpha }(t)\zeta_i
=F_{j}^{-\alpha }(t)\zeta_j \text{ and }F_{i}(t)\neq 0\right)=0,
\end{align*}
where $\left( \zeta_i,\zeta_j \right) $ are independent and
distributed as $\zeta$, independently of $F(t)$.  Clearly, this is
satisfied if the distribution of $\zeta $ has no atoms. Now recall
that we are in the case where $\nu $ is infinite and suppose that
there exists $t>0$ such that $\mathbb{P(\zeta}=t\mathbb{)}>0$. Recall
also that, conditional on $u<\zeta$, $\zeta=u+\sup_{i \geq
  1}F_i(u)^{-\alpha} \zeta_i $ where $(\zeta_i, i \geq 1)$ are
independent copies of $\zeta$, independent of $F(u)$.  Moreover, the
supremum is actually a maximium, since we know there exists $x \in
(0,1)$ such that $G(x)=\zeta$.  Then for all $0 < u < t$,
\begin{align*}
&\mathbb{P}\left( \exists i:F_{i}^{-\alpha }(u)\zeta_i =t-u\right) >0 \\
&\Leftrightarrow \exists i:\mathbb{P}\left( F_{i}^{-\alpha }(u)\zeta
=t-u\right) >0\text{ (with }\zeta \text{ independent of }F(u)\text{)} \\
&\Leftrightarrow \mathbb{P}\left( \lambda ^{-\alpha }(u)\zeta =t-u\right) >0,
\end{align*}
where $\lambda $ denotes the tagged fragment process. 
Recall that $\lambda(u) = \exp(-\xi(\rho(u)))$, where $\xi$ is a subordinator with Laplace exponent given by (\ref{eqn:Laplacetagged}).  Now, for any $b > 0$,
\[
\Prob{\xi(\rho(u)) = b} \leq \Prob{\exists v \geq 0: \xi(v) = b}.
\]
But we know from Kesten's theorem (Proposition 1.9 in
Bertoin~\cite{BertoinStFl}) that the right-hand side is 0 because the
L\'evy measure of the subordinator $\xi$ is infinite and it has no
drift.
 Hence, $\mathbb{P}\left( \lambda ^{-\alpha }(u)\zeta
  =t-u\right) =0$ for all $0<u<t$, and we can deduce the claimed
uniqueness.
\end{proof}

\bigskip
\begin{proof}[Proof of Corollary \ref{corologstable}.] It has been
proved in Haas and Miermont \cite[Section 3.5]{Haas/Miermont} that the
dislocation measure $\nu$ of any stable fragmentation satisfies 
\[
\int_{\Sfl} s_1^{-1}\I{s_1<1/2} \nu (\mathrm d \mathbf s)<\infty.
\]
(Note that this is obvious in the Brownian case since the
fragmentation is binary and so $\nu(s_1<1/2)=0$.)  From \cite[Section
4.4]{Haas/Miermont}, we know that the parameter $\gamma_\nu$ (defined
in Theorem \ref{thm:largestfrag} (ii)) associated with the
dislocation measure $\nu$ of the stable fragmentation with index $\alpha$ is given
by $\gamma_{\nu}=-\alpha$. Hence, both assumptions of Theorem
\ref{thm:largestfrag} (i) and (ii) are satisfied.
\end{proof}

%%%%%%%%%%%%%%%%%%%
% Acknowledgments %
%%%%%%%%%%%%%%%%%%%

\noindent \textbf{Acknowledgments.}  We are grateful to Jean Bertoin
and Gr\'egory Miermont for helpful discussions.
C.~G.\ would like to
thank the SPINADA project for funding which facilited this
collaboration, and Pembroke College, Cambridge for its support during
part of the project.  C.~G.\ was funded by EPSRC Postdoctoral
Fellowship EP/D065755/1.

%%%%%%%%%%%%%%%%
% Bibliography %
%%%%%%%%%%%%%%%%

\bibliographystyle{abbrv}

\bibliography{frag}

\end{document}